\documentclass[11pt]{article}
\usepackage{amsmath,amssymb,amsthm,enumerate,url}
\usepackage[margin=1in]{geometry}

\newtheorem{thm}{Theorem}[section]
\newtheorem*{thm*}{Theorem}
\newtheorem{lemma}[thm]{Lemma}
\newtheorem*{lemma*}{Lemma}

\newtheorem*{prop*}{Proposition}
\newtheorem{cor}[thm]{Corollary}

\newtheorem{fact}[thm]{Fact}
\newtheorem*{not*}{Notation}
\newtheorem{claim}[thm]{Claim}
\newtheorem*{claim*}{Claim}
\newtheorem*{fact*}{Fact}

\newtheorem{dfn}[thm]{Definition}

\newtheorem{conj}[thm]{Conjecture}

\newcommand{\Q}{\mathbb{Q}}
\newcommand{\Z}{\mathbb{Z}}

\newcommand{\R}{\mathbb{R}}
\newcommand{\C}{\mathbb{C}}
\newcommand{\LL}{\mathcal{L}}

\newcommand{\B}{\mathcal{B}}
\def\PP{\mathcal{P}}
\DeclareMathOperator{\tr}{Tr}
\DeclareMathOperator\LIS{LIS}
\DeclareMathOperator{\sign}{sign}

\DeclareMathOperator{\Ex}{\mathbb{E}}

\newcommand{\eps}{\varepsilon}

\newcommand{\Span}{\mathrm{Span}}

\newcommand{\lcm}{\mathrm{lcm}}

\renewcommand{\hat}[1]{\widehat{#1}}
\newcommand{\choosesq}[2]{\genfrac{[}{]}{0pt}{}{#1}{#2}}
\newcommand{\ip}[1]{\langle #1 \rangle}

\newcommand{\ignore}[1]{}

\renewcommand{\Pr}{\mathbb{P}}

\DeclareMathOperator{\vol}{Vol}

\begin{document}

\title{Probabilistic existence of regular combinatorial structures\thanks{An extended abstract \cite{KuperbergLovettPeled12} of this paper appeared in the 44th ACM Symposium on Theory of Computing (STOC 2012).}}

\author{Greg Kuperberg \thanks{University of California, Davis.  E-mail:
    \texttt{greg@math.ucdavis.edu}. Supported by NSF grant CCF-1013079.}
    \and Shachar Lovett \thanks{Institute for Advanced Study.  E-mail:
    \texttt{slovett@math.ias.edu}.  Supported by NSF grant DMS-0835373.}
    \and Ron Peled \thanks{Tel Aviv University, Israel.
    E-mail: \texttt{peledron@post.tau.ac.il}. Supported by an ISF
    grant and an IRG grant.}}

\maketitle

\begin{abstract} We show the existence of regular combinatorial objects
which previously were not known to exist. Specifically, for a wide
range of the underlying parameters, we show the existence of
non-trivial orthogonal arrays, $t$-designs, and $t$-wise
permutations.  In all cases, the sizes of the objects are optimal up
to polynomial overhead. The proof of existence is probabilistic.  We
show that a randomly chosen structure has the required properties
with positive yet tiny probability. Our method allows also to give
rather precise estimates on the number of objects of a given size
and this is applied to count the number of orthogonal arrays,
$t$-designs and regular hypergraphs. The main technical ingredient
is a special local central limit theorem for suitable lattice random
walks with finitely many steps.
\end{abstract}

\newpage
\tableofcontents
\newpage

\section{Introduction}\label{sec:intro}

We introduce a new framework for establishing the existence of
regular combinatorial structures, such as orthogonal arrays,
$t$-designs and $t$-wise permutations.  Let $B$ be a finite set and
let $V$ be a vector space of functions from $B$ to the rational
numbers $\Q$.  We study when there is a small subset $T\subset B$
satisfying
\begin{equation} \label{eq:prob_def}
    \frac{1}{|T|}\sum_{t\in T} f(t) = \frac{1}{|B|}\sum_{b\in B}
    f(b)\quad\text{for all $f$ in $V$.}
\end{equation}
In probabilistic terminology, equation \eqref{eq:prob_def} means that if
$t$ is a uniformly random element in $T$ and $b$ is a uniformly random
element in $B$ then
\begin{equation}\label{eq:prob_def_prob}
    \Ex_{t \in T}[f(t)] = \Ex_{b \in B}[f(b)]\quad\text{for all $f$ in $V$,}
\end{equation}
where $\Ex$ denotes expectation. Of course, \eqref{eq:prob_def}
holds trivially when $T=B$. Our goal is to find conditions on $B$
and $V$ that yield a small subset $T$ that satisfies
\eqref{eq:prob_def}, where in our situations, small will mean
polynomial in the dimension of $V$. We remark that in many natural
problems one might encounter a vector space $V$ over $\R$ or $\C$
instead. However, since \eqref{eq:prob_def} is a rational equation,
we can always reduce to the case of rational vector spaces.

A more concrete realization of the above framework is given by the
following problem. Let $\phi$ be a matrix with rational entries
whose rows are indexed by a finite set $B$ and whose columns are
indexed by a finite set $A$. When is there a small subset $T\subset
B$ such that the average of the rows indexed by $T$ equals the
average of all rows? This problem is a special case of the above
framework with $V$ being the subspace spanned by the columns of
$\phi$. In fact, the general framework can always be reduced to such
a problem by choosing a basis $(\phi_a)$, $a\in A$, of $V$ and
defining the matrix $\phi$ by $\phi(b,a)=\phi_a(b)$.

Our main theorem, Theorem~\ref{thm:main}, gives sufficient
conditions for the existence of a small subset $T$ satisfying
\eqref{eq:prob_def}. A second theorem,
Theorem~\ref{thm:count_solutions}, provides sharp estimates on the
number of such subsets of a given size. We apply the theorems to
establish results in three interesting cases of the general
framework: orthogonal arrays, $t$-designs, and $t$-wise
permutations.  These are defined and discussed in more detail in the
next sections. Our methods solve an open problem, whether there
exist non-trivial $t$-wise permutations for every $t$. They
strengthen Teirlinck's theorem \cite{Teirlinck87}, which was the
first theorem to show the existence of $t$-designs for every $t$.
And they improve existence results for orthogonal arrays, when the
size of the alphabet is divisible by many distinct primes. Moreover,
in all three cases considered, we show the existence of a structure
whose size is optimal up to polynomial overhead. In addition, we
provide sharp estimates for the number of orthogonal arrays and
$t$-designs of a given size. As a special case, these yield
estimates for the number of regular hypergraphs of a given degree.

Our approach to the problem is via probabilistic arguments. In
essence, we prove that a random subset of $B$ satisfies equation
\eqref{eq:prob_def} with positive, albeit tiny, probability. Thus
our method is one of the few known methods for showing existence of
rare objects. This class includes such other methods as the Lov\'asz
local lemma \cite{ErdosLovasz73} and Spencer's ``six deviations
suffice'' method \cite{Spencer1985}. However, our method does not
rely on these previous approaches.  Instead, our technical
ingredient is a special version of the (multi-dimensional) local
central limit theorem with finitely many steps. We cannot use any
``off the shelf'' local central limit theorem, not even one enhanced
by a Berry-Esseen-type estimate of the rate of convergence, since
the number of steps of our random walk is small compared to the
dimension of the space in which it takes its values. Instead, we
prove the local central limit theorem that we need directly using
Fourier analysis. Section~\ref{sec:proof_overview} gives an overview
of our approach.

We also mention that efficient randomized algorithm versions of the
Lov\'asz local lemma \cite{Moser2009, MoserTardos2010} and Spencer's
method \cite{Bansal2010,LovettMeka12} have recently been found.  Relative to
these new algorithms, the objects that they produce are no longer rare.
Our method is the only one that we know that shows the existence of rare
combinatorial structures, which are still rare relative to any known,
efficient, randomized algorithm.

\subsection{Orthogonal arrays}
\label{sec:OAs}

Here and in the rest of the paper we use the notation
$[m]:=\{1,\ldots, m\}$. A subset $T \subset [q]^n$ is an
\emph{orthogonal array of alphabet size $q$, length $n$ and strength
$t$} if it yields all strings of length $t$ with equal frequency if
restricted to any $t$ coordinates. In other words, for any distinct
indices $i_1,\ldots,i_t \in [n]$ and any (not necessarily distinct)
values $v_1,\ldots,v_t \in [q]$,
\begin{equation}\label{eq:OA_def}
\left|\{x \in T: x_{i_1}=v_1,\ldots,x_{i_t}=v_t\}\right| = q^{-t} |T|.
\end{equation}
Equivalently, choosing $x=(x_1,\ldots,x_n) \in T$ uniformly, the
distribution of each coordinate of $x$ is uniform in $[q]$ and every
$t$ coordinates of $x$ are independent ($x$ is $t$-wise
independent). For an introduction to orthogonal arrays
see~\cite{orthogonalarrays}.

Orthogonal arrays fit into our general framework as follows. We take $B$
to be $[q]^n$ and $V$ to be the space spanned by all functions of the form
\begin{equation*}
  f_{(I,v)}(x_1,\ldots, x_n) = \begin{cases}
    1& x_i=v_i\text{ for all $i\in I$}\\0&\text{Otherwise}
  \end{cases},
\end{equation*}
with $I\subset[n]$ a subset of size $t$ and $v\in[q]^I$.  With this choice,
a subset $T\subset B$ satisfying \eqref{eq:prob_def} is precisely an
orthogonal array of alphabet size $q$, length $n$ and strength $t$.

It is well known that if $T \subset [q]^n$ is an orthogonal array of
strength $t$ then $|T| \ge \left(\frac{cqn}{t}\right)^{t/2}$ for
some universal constant $c>0$ (see, e.g.,~\cite{Rao73}). Matching
constructions of size $|T|\le q^{c t}\left(\frac{n}{t}\right)^{c_q
t}$ are known, however, as these rely on finite field properties,
the constant $c_q$ generally tends to infinity with the number of
distinct prime factors of $q$. Our technique provides the first
upper bound on the size of orthogonal arrays in which the constant
in the exponent is independent of $q$. Here and below, a universal
constant is a constant independent of all other parameters.

\begin{thm}[Existence of orthogonal arrays]\label{thm:OA_first_version}
For all integers $q\ge 2$, $n\ge 1$ and $1\le t\le n$ there exists
an orthogonal array $T$ of alphabet size $q$, length $n$ and
strength $t$ satisfying $|T| \le \left(\frac{cqn}{t}\right)^{ct}$
for some universal constant $c>0$.
\end{thm}
Moreover, we provide a rather precise count of the number of
orthogonal arrays with given parameters.
\begin{thm}[Number of orthogonal arrays]\label{thm:OA_count}
There exists a constant $c>0$ such that for all integers $q\ge 2$,
$n\ge 1$ and $1\le t\le n$ and for all $N$ satisfying that $N$ is a
multiple of $q^t$ and $\min(N, q^n-N)\ge
\left(\frac{cqn}{t}\right)^{ct}$, we have that the number of
orthogonal arrays $T$ of alphabet size $q$, length $n$, strength $t$
and $|T|=N$ equals
\begin{equation*}
\frac{q^{-\frac{1}{2}n\binom{n-1}{t}(q-1)^t}}{(2\pi
p(1-p))^{\frac{1}{2}\sum_{i=0}^t
\binom{n}{i}(q-1)^i}p^N(1-p)^{q^n-N}} (1+\delta)
\end{equation*}
where $p:=\frac{N}{q^n}$ and $|\delta|\le
\frac{\left(\frac{cqn}{t}\right)^{ct}}{\sqrt{\min(N,q^n-N)}}$.
\end{thm}
The case $q=2$ of this theorem has appeared in the work of Canfield,
Gao, Greenhill, McKay and Robinson \cite{CGGMR2010} with slightly
more restrictive assumptions. All other cases of the theorem appear
to be new.

As a final remark we note that any orthogonal array $T$ of alphabet
size $q$ and strength $t$ must satisfy that $|T|$ is a multiple of
$q^t$ by \eqref{eq:OA_def}. Hence the essential restriction on $N$
in the last theorem is only that $N$ be bounded away from $0$ and
$q^n$.

\subsection{Designs}
\label{sec:designs}

A (simple) \emph{$t$-$(v,k,\lambda)$ design}, or $t$-design for
short, is a family of distinct subsets of $[v]$, where each set is
of size $k$, such that each $t$ elements belong to exactly $\lambda$
sets. In other words, denoting by $\choosesq{v}{k}$ the family of
all subsets of $[v]$ of size $k$, a set $T \subset \choosesq{v}{k}$
is a $t$-design if for any distinct elements $i_1,\ldots,i_t \in
[v]$,
\begin{equation}\label{eq:intro:design_def}
\left|\{s \in T: i_1,\ldots,i_t \in s\}\right|
    = \lambda.
\end{equation}
It follows that $\lambda$ satisfies the relation
\begin{equation}\label{eq:intro:lambda_rel}
  \lambda = \frac{\binom{k}{t}}{\binom{v}{t}} |T|.
\end{equation}
For an introduction to combinatorial designs see~\cite{Hand07}.

Our general framework includes $t$-designs as follows. We take $B$ to
be $\choosesq{v}{k}$ and $V$ to be the space spanned by all functions of
the form
\begin{equation*}
  f_{a}(b) = \begin{cases}
    1& a\subset b\\0&\text{Otherwise}
  \end{cases},
\end{equation*}
with $a\in \choosesq{v}{t}$. With this choice, a subset $T\subset B$
satisfying \eqref{eq:prob_def} is precisely a simple
$t$-$(v,k,\lambda)$ design, with $\lambda$ given by
\eqref{eq:intro:lambda_rel}.

Although $t$-designs have been investigated for many years, the
basic question of existence of a design for a given set of
parameters $t,v,k$ and $\lambda$ remains mostly unanswered unless
$t$ is quite small (see remark below for recent progress). The case
$t=2$ is known as a block design and much more is known about it
than for larger $t$. Explicit constructions of $t$-designs for $t
\ge 3$ are known for various specific constant settings of the
parameters (e.g. $5$-$(12,6,1)$ design). The breakthrough result of
Teirlinck~\cite{Teirlinck87} was the first to establish the
existence of non-trivial $t$-designs for $t\ge 7$. In Teirlinck's
construction, $k=t+1$ and $v$ satisfies congruences that grow very
quickly as a function of $t$. Other sporadic and infinite examples
have been found since then (see~\cite{Hand07} or \cite{Magliveras09}
and the references within), however, the set of parameters which
they cover is still very sparse.

It follows from \eqref{eq:intro:lambda_rel} that any
$t$-$(v,k,\lambda)$ design $T$ has size $|T|=\lambda \binom{v}{t}/
\binom{k}{t} \ge \left(\frac{v}{k}\right)^{t}$. Moreover, it can be
shown \cite{RayWilson75} that whenever $v \ge k+t$ the inequality
$|T| \ge {v \choose \lfloor t/2 \rfloor} \ge
\left(\frac{v}{t}\right)^{\lfloor t/2 \rfloor}$ holds. Even when
existence has been shown, the designs obtained are often inefficient
in the sense that their size is much larger than these lower bounds
permit. One of the main results of our work is to establish the
existence of efficient $t$-designs for a wide range of parameters.

\begin{thm}[Existence of $t$-designs]\label{thm:designs_first_version}
For all integers $v\ge 1$, $1\le t\le v$ and $t \le k \le v$ there
exists a $t$-$(v,k,\lambda)$ design whose size is at most
$\left(\frac{cv}{t}\right)^{c t}$ for some universal constant $c>0$.
\end{thm}

Our work also provides a rather precise count of the number of
$t$-designs of a given size with given parameters. To state this
count precisely, we recall the well-known observation that if $T$ is
a $t$-$(v,k,\lambda)$ design then for every $1\le s\le t$, each
subset of size $s$ in $[v]$ is covered by exactly
\begin{equation*}
  \lambda_s := \frac{\binom{k}{s}}{\binom{v}{s}} |T|
\end{equation*}
sets in $T$. In particular, $(\lambda_s)$, $1\le s\le t$ must be
integers. Our next theorem, in addition to estimating the number of
designs, implies that if $|T|$ (or equivalently $\lambda$) is
sufficiently large, these integrality conditions suffice for the
existence of $t$-$(v,k,\lambda)$ designs.
\begin{thm}[Number of $t$-designs]\label{thm:counting_designs}
There exists a constant $c>0$ such that for all integers $v\ge 1$,
$1\le t\le v$ and $t \le k \le v$ and for all $N$ satisfying that
the numbers
\begin{equation*}
  \frac{\binom{k}{s}}{\binom{v}{s}} N \text{ are integers for $1\le s\le
  t$}
\end{equation*}
and satisfying that $\min(N, \binom{v}{k}-N)\ge
\left(\frac{cv}{t}\right)^{c t}$, we have that the number of
$t$-$(v,k,\lambda)$ designs $T$ of size $|T|=N$ equals
\begin{equation*}
\frac{1}{(2\pi
p(1-p))^{\frac{1}{2}\binom{v}{t}}p^N(1-p)^{\binom{v}{k}-N}}\left(\prod_{s=0}^t
\left[\frac{\binom{k-s}{k-t}}{\binom{v-t-s}{k-t}}\right]^{\frac{1}{2}\left(\binom{v}{s}-\binom{v}{s-1}\right)}\right)
(1+\delta)
\end{equation*}
where $p:=\frac{N}{\binom{v}{k}}$, $|\delta|\le
\frac{\left(\frac{cv}{t}\right)^{c
t}}{\sqrt{\min(N,\binom{v}{k}-N)}}$ and $\binom{v}{-1}$ is defined
to be 0.
\end{thm}

{\bf Remark.} After the first version of this work appeared, Peter
Keevash published his breakthrough work~\cite{keevash2014existence}
proving the existence of Steiner systems, combinatorial designs with
$\lambda=1$, as well as designs with larger values of $\lambda$.
Quite recently, Glock, K{\"u}hn, Lo and
Osthus~\cite{glock2016existence} gave a new proof of this
fundamental result. The techniques used there are quite different
from ours, Keevash employing `randomised algebraic constructions'
and Glock et al. using the method of `iterative absorption'. These
powerful techniques have so far been limited to the case that the
parameters $k, t$ are fixed, or growing very slowly compared with
$v$. Thus, the ranges of parameters covered by these theorems and
our own are mostly complementary: On the one hand, we require that
\begin{equation}\label{eq:inequality_on_lambda}
  \left(\frac{cv}{t}\right)^{c t}\le \lambda\le
\binom{v-t}{k-t}-\left(\frac{cv}{t}\right)^{c t}
\end{equation}
for some absolute constant $c>0$, excluding the fundamental case of
Steiner systems and other small $\lambda$ designs
($\binom{v-t}{k-t}$ is the maximal value for $\lambda$, obtained for
the complete design $T = \choosesq{v}{k}$). On the other hand, our
theorem applies to arbitrary tuples $t-(v,k,\lambda)$ satisfying the
necessary congruence conditions and the inequality
\eqref{eq:inequality_on_lambda}. In addition, our method provides
the exact asymptotics of the number of designs of a given size as
this size grows, whereas the techniques of
\cite{keevash2014existence} and \cite{glock2016existence} have so far yielded less precise control \cite{keevash2015counting}.

\subsubsection{Regular hypergraphs}
A (simple) \emph{$k$-uniform hypergraph on $n$ vertices} is a family
of sets of size $k$ (called edges) on $n$ elements (called
vertices). A hypergraph is \emph{$d$-regular} if each vertex belongs
to exactly $d$ edges. It is straightforward to check that
$d$-regular, $k$-uniform hypergraphs are the same as $1$-$(n,k,d)$
designs.

The existence question for $d$-regular, $k$-uniform hypergraphs is
quite simple, such a hypergraph exists if and only if $nd$ is
divisible by $k$, and in this case any such hypergraph has exactly
$\frac{nd}{k}$ edges. However, counting the number of $d$-regular,
$k$-uniform hypergraphs is a non-trivial problem which has received
much attention in the literature, mainly in the graph case (k=2);
See the paper by McKay and Wormald \cite{MckayWormald} and
references within. In the graph case, until very recently,
approximate counts were known only when the graphs are either
somewhat sparse, or very dense. This gap has now been filled in the
work of Liebenau and Wormald \cite{liebenau2017asymptotic}.

Since $d$-regular, $k$-uniform hypergraphs are a special case of
$t$-designs, we may translate Theorem~\ref{thm:counting_designs} to
obtain a count of such hypergraphs. Our result applies for $k$ and
$d$ sufficiently large and appears to be new.
\begin{thm}
There exists a constant $c>0$ such that for all integers $n\ge 2$,
$1\le k\le n$ and $1\le d\le \binom{n-1}{k-1}$ satisfying that $nd$
is divisible by $k$ and $\min(\frac{nd}{k},
\binom{n}{k}-\frac{nd}{k})\ge n^c$, we have that the number of
$d$-regular, $k$-uniform hypergraphs on $n$ vertices equals
\begin{equation*}
\frac{1}{(2\pi
p(1-p))^{\frac{n}{2}}p^{\frac{nd}{k}}(1-p)^{\binom{n}{k}-\frac{nd}{k}}}\left(\frac{k}{\binom{n-1}{k-1}}\right)^{\frac{1}{2}}\left(\frac{1}{\binom{n-2}{k-1}}\right)^{\frac{1}{2}(n-1)}(1+\delta)
\end{equation*}
where $p:=\frac{nd}{k\binom{n}{k}}$ and $|\delta|\le
\frac{n^c}{\sqrt{\min(\frac{nd}{k},\binom{n}{k}-\frac{nd}{k})}}$.
\end{thm}

We mention also a related result on the asymptotic formula for the
number of \emph{binary contingency tables} obtained by Canfield and
McKay \cite{CanfieldMckay}. A binary contingency table is an
$N\times n$ matrix with entries in $\{0,1\}$, all row sums equal to
$k$ and all column sums equal to $d$ (so that necessarily
$N=\frac{nd}{k}$). Such a matrix describes a $d$-regular,
$k$-uniform hypergraph on $n$ vertices with \emph{labelled} edges
and allowing multiple edges. Equivalently, it describes a bipartite
graph with $N$ vertices of degree $k$ on one side and $n$ vertices
of degree $d$ on the other side. When the parameters $(N,n,k,d)$ are
such that most binary contingency tables have all rows distinct, the
number of such tables is close to $N!$ times the number of
$d$-regular, $k$-uniform hypergraphs on $n$ vertices. Thus our
asymptotic formula and the formula of \cite{CanfieldMckay} should be
related in certain ranges of the parameters. We do not develop this
direction here.

\subsection{Permutations}
\label{sec:perms}

A family of permutations $T \subset S_n$ is called a \emph{$t$-wise
permutation} if its action on any $t$-tuple of elements is uniform.
In other words, for any distinct elements $i_1,\ldots,i_t \in [n]$ and
distinct elements $j_1,\ldots,j_t \in [n]$,
\begin{equation}\label{eq:perm_def}
    \left| \{\pi \in T: \pi(i_1)=j_1,\ldots,\pi(i_t)=j_t\} \right| =
    \frac{1}{n(n-1)\cdots(n-t+1)} |T|.
\end{equation}

Our general framework includes $t$-wise permutations as follows. We take
$B=S_n$ and $V$ to be the space spanned by all functions of the form
\begin{equation*}
  f_{(i, j)}(b) = \begin{cases}
    1& b(i_1)=j_1,\ldots,b(i_t)=j_t\\0&\text{Otherwise}
  \end{cases},
\end{equation*}
where $i=(i_1,\ldots,i_t)$ and $j=(j_1,\ldots,j_t)$ are $t$-tuples of
distinct elements in $[n]$. With this choice, a subset $T\subset B$
satisfying \eqref{eq:prob_def} is precisely a $t$-wise permutation.

Equation~\eqref{eq:perm_def} yields a lower bound on the size of
$t$-wise permutations, $|T|\ge n(n-1)\cdots(n-t+1)$. Constructions
of families of $t$-wise permutations matching this lower bound are
known for $t=1,2,3$: the group of cyclic shifts $x \mapsto x+a$
modulo $n$ is a $1$-wise permutation; the group of invertible affine
transformations $x \mapsto ax+b$ over a finite field $\mathbb{F}$
yields a $2$-wise permutation; and the group of M\"{o}bius
transformations $x \mapsto (ax+b)/(cx+d)$ with $ad-bc=1$ over the
projective line $\mathbb{F} \cup \{\infty\}$ yields a $3$-wise
permutation. However, it is known (see, e.g., \cite{Cameron95},
Theorem 5.2) that for $n \ge 25$ and $t \ge 4$ there are no
\emph{subgroups} of $S_n$ which form a $t$-wise permutation (of any
size), other than $S_n$ itself and the alternating group $A_n$.
Moreover, for $t \ge 4$ (and $n$ large enough), no non-trivial
constructions of $t$-wise permutations are known at all
\cite{KaplanNaorReingold05,AlonLovett11}, with the exception of the
recent work \cite{FinucanePeledYaari} which constructs rather large,
but non-trivial, $t$-wise permutations of size $t^{2n}$ for
infinitely many values of $n$ and $t$. One of our main results is
the existence of small $t$-wise permutations for all $n$ and $t$.

\begin{thm}[Existence of $t$-wise
permutations]\label{thm:perms_first_version} For all integers $n\ge 1$
and $1\le t\le n$ there exists a $t$-wise permutation $T\subset S_n$
satisfying $|T| \le (cn)^{c t}$ for some universal constant $c>0$.
\end{thm}
We leave the problem of estimating the number of $t$-wise
permutations of a given size for future work. See
Section~\ref{sec:number_t_wise_permutations} where the problem is
reduced to the calculation of a determinant and certain numerical
calculations are presented.

\subsection{Proof overview}
\label{sec:proof_overview}

The idea of our approach is as follows. Let us first consider the
following slight simplification of our main idea. Let $T$ be a
random multiset of $B$ of some fixed size $N$ chosen by sampling $B$
uniformly and independently $N$ times (with replacement).  Let
$(\phi_a)_{a\in A}$ be a basis of integer-valued functions for $V$
(where $A$ is some arbitrary finite index set). Observe that $T$
satisfies \eqref{eq:prob_def} if and only if
\begin{equation}\label{eq:phi_a_expectation}
  \sum_{t\in T} \phi_a(t) = \frac{N}{|B|}\sum_{b\in B}\phi_a(b)
  = \Ex\left[\sum_{t\in T} \phi_a(t)\right] \quad
  \text{for all $a$ in $A$},
\end{equation}
where we add terms multiple times if they appear in $T$ multiple
times. Thus, defining an integer-valued random variable
\begin{equation*}
  X_a:=\sum_{t\in T} \phi_a(t)
\end{equation*}
and $X:=(X_a)_{a\in A} \in \Z^A$ we see that existence of a multiset
of size $N$ satisfying \eqref{eq:prob_def} will follow if we can
show that $\Pr[X=\Ex[X]]>0$. To this end we examine more closely the
distribution of $X$. Let $t_1,\ldots, t_N$ be the random elements
chosen in forming $T$. The basis $(\phi_a)_{a\in A}$ defines a
mapping $\phi:B\to\Z^A$ by the trivial
\begin{equation*}
  \phi(b)_a:=\phi_a(b).
\end{equation*}
Observe that our choice of random model implies that the vectors
$(\phi(t_i))_{i\in [N]}$ are independent and identically distributed. Hence,
\begin{equation}\label{eq:X_as_sum}
    X = \sum_i \phi(t_i)
\end{equation}
may be viewed as the end position of an $N$-step random walk in the
lattice $\Z^{A}$. Thus we may hope that if $N$ is sufficiently
large, then $X$ has an approximately (multi-dimensional) Gaussian
distribution by the central limit theorem. If the relevant local
central limit theorem holds as well, then the probability $\Pr[X=x]$
also satisfies a Gaussian approximation. In particular, since a
(non-degenerate) Gaussian always has positive density at its
expectation, we could conclude that $\Pr[X=\Ex[X]]>0$ as desired.
Moreover, to estimate the number of multisets of size $N$ satisfying
\eqref{eq:prob_def} we need only estimate $\Pr[X=\Ex[X]]$ using the
Gaussian approximation.

The above description is the essence of our approach. The main
obstacle is, of course, pointed out in the last step. We must
control the rate of convergence of the local central limit theorem
well enough so that the convergence error does not outweigh the
probability density of the Gaussian distribution at $\Ex[X]$. Recall
that the order of magnitude of such a density is typically
$c^{-|A|}$ for some constant $c>1$, and recall that $|A|$ is the
dimension of $V$, which is the main parameter of our problem.  So we
indeed have very small probabilities.  For this reason, and because
we want convergence when $N$ is only polynomial in the dimension of
$V$, we were unable to use any standard local central limit theorem.
Instead, we develop an ad hoc version using Fourier analysis.

In our proof of the main theorem, we modify the above description in
one respect.  It is technically more convenient to work with a
slightly different probability model.  Instead of choosing $T$ as
above, we set $p:=N/|B|$ and define $T$ by taking each element of
$B$ into $T$ independently with probability $p$.  This has the
benefit of guaranteeing that $T$ is a proper set instead of a
multiset. However, it has also the disadvantage that it does not
guarantee that $|T|=N$. To remedy this, we assume that the space $V$
contains the constant function $h \equiv 1$; or if not, we can add
it to $V$ at the minor cost of increasing the dimension of $V$ by 1.
With this assumption, since
\begin{equation*}
  \Ex\bigg[\sum_{t\in T} h(t)\bigg] = \Ex[|T|] = N,
\end{equation*}
we see that \eqref{eq:phi_a_expectation}, or equivalently
$X=\Ex[X]$, implies both that $|T|=N$ and that~\eqref{eq:prob_def}
holds. Another disadvantage is that in this new probability model,
the vector $X$ is no longer a sum of identically distributed
variables. However, since the summands in \eqref{eq:X_as_sum} are
still independent, we can continue to use Fourier analysis methods
in our proof.

We cannot expect there to always be a small subset $T\subset B$ that
satisfies \eqref{eq:prob_def}. For instance, Alon and
Vu~\cite{AlonVu97} found a regular hypergraph with $n$ vertices and
$\approx n^{n/2}$ edges, with no non-trivial regular sub-hypergraph.
Here, a regular hypergraph is one in which every vertex belongs to
the same number of hyperedges. We may describe their example in our
language by letting $B$ be the set of hyperedges of this hypergraph,
$A$ be its vertex set, and define $\phi:B\to\{0,1\}^A$ by letting
$\phi(b)$ be the indicator function of the set of vertices contained
in $b$. The result of \cite{AlonVu97} implies that while the vector
$\sum_{b\in B} \phi(b)$ is constant, this property is not shared by
$\sum_{t\in T} \phi(t)$ for any non-empty, proper subset $T\subset
B$. Thus, we need to impose certain conditions on $B$ and $V$, or
equivalently on the map $\phi$.

We will require certain divisibility, boundedness and symmetry
assumptions. Our main theorem shows that these conditions suffice to
yield the existence of a small solution of \eqref{eq:prob_def} and,
moreover, a rather precise estimate for the number of solutions. The
conditions and the statement of the theorem appear in
Section~\ref{sec:framework}. The existence and counting theorems for
orthogonal arrays, $t$-designs and $t$-wise permutations follow by
verifying these conditions for the choice of $B$ and $V$ detailed in
Sections~\ref{sec:OAs} through \ref{sec:perms}.

\subsection{Related work}

In the probabilistic formulation \eqref{eq:prob_def_prob} of our
problem we seek a small subset $T \subset B$ such that the uniform
distribution over $T$ simulates the uniform distribution over $B$
with regards to certain tests. Equivalently, we want to estimate the
probability that a random walk in a lattice with a prescribed set of
allowed steps reaches a specific point. Our approach via Fourier
transform to this problem is certainly classical and there have been
several works using this idea to enumerate combinatorial objects.
Besides the works \cite{CGGMR2010}, \cite{MckayWormald} and
\cite{CanfieldMckay} cited above, we mention de Launey and Levin's
enumeration of partial Hadamard matrices \cite{de2010fourier},
Montgomery's enumeration of the incidence matrices of balanced
incomplete block designs \cite{montgomery2014asymptotic} and
Barvinok and Hartigan's enumeration of the number of contingency
tables with prescribed sums \cite{barvinok2010maximum,
barvinok2012asymptotic, barvinok2013number} and the number of graphs
with a given `tame' degree sequence \cite{barvinok2013number}. For
further related developments we refer to the paper of Isaev and
McKay \cite{isaev2016complex} and the recent book of Barvinok
\cite{barvinok2017book}.

Our work introduces general theorems, detailed in
Section~\ref{sec:framework} below, to treat existence and
enumeration problems which fall in the framework
\eqref{eq:prob_def}. We are aware of only one such general framework
in the previous literature, the approach of Barvinok and Hartigan
\cite[Theorem~2.6]{barvinok2010maximum} for the equivalent problem
of counting the number of vectors with $0,1$ coordinates inside a
given polytope. The conditions required to apply the theorem of
\cite{barvinok2010maximum} involve the bounding of certain quadratic
forms, while our approach relies on symmetry and ideas arising from
coding theory which manifest in the `low density parity check'
(LDPC) condition (that $V^{\perp}$ has a $c_3$-bounded integer basis
in $\ell_1$). The LDPC condition plays a central role in the proofs
of our main theorems, most prominently in the proof of
Theorem~\ref{thm:perms_first_version} on $t$-wise permutations where
it integrates naturally with the representation theory of the
symmetric group. We have not seen the problem related to coding
theory in previous works and believe this relation will be of value
in other contexts.

A framework for studying problems similar to ours in a continuous
setup was introduced by Seymour and Zaslavsky
\cite{SeymourZaslavsky84}. They consider the case when $B$ is a
path-connected topological space with a measure $\mu$ of full
support, and the problem of finding a finite subset $T\subset B$
such that the uniform distribution over $T$ integrates certain
continuous functions exactly as $\mu$. This framework is sometimes
referred to as ``averaging sets'', ``equal-weight quadratures'' or
``Chebyshev-type quadratures''. It was shown in
\cite{SeymourZaslavsky84} that such $T$ exist in great generality
and that their cardinality can be any number with finitely many
exceptions. More recently, Kane~\cite{kane2015small} gave effective
bounds on the size of $T$ in terms of certain symmetries of $B$ (see
also \cite{BondarenkoRadchenkoViazovska} for the case of spherical
designs and \cite{Gilboa2016Chebyshev} for the classical case of
integrating polynomials against a one-dimensional measure).

Next, there are two ways to relax the problem we consider and make
its solution easier. Each one raises new questions regarding
explicit solutions.

One relaxation is to allow a set $T$ with a \emph{non-uniform}
distribution $\mu$ which simulates the \emph{uniform} distribution
over $B$. For many practical applications of $t$-designs and
$t$-wise permutations in statistics and computer science, but not
quite every application, this relaxation is as good as the uniform
question. The existence of a solution with small support is
guaranteed by Carath\'eodory's theorem, using the fact that the
constraints on $\mu$ are all linear equalities and inequalities.
Moreover, such a solution can be found efficiently, as was shown by
Karp and Papadimitriou~\cite{KarpPapadimitriou82} and in more
general settings by Koller and Megiddo~\cite{KollerMegiddo94}. Alon
and Lovett~\cite{AlonLovett11} give a strongly explicit analog of
this in the case of $t$-wise permutations and more generally in the
case of group actions.

A different relaxation is to require the uniform distribution over
$T$ to only approximately satisfy equation \eqref{eq:prob_def_prob}.
Then it is trivial that a sufficiently large random subset $T
\subset B$ satisfies the requirement with high probability, and the
question is to find an explicit solution. For instance, we can relax
the problem of $t$-wise permutations to \emph{almost} $t$-wise
permutations.  For this variant an optimal solution (up to
polynomial factors) was achieved by Kaplan, Naor and
Reingold~\cite{KaplanNaorReingold05}, who gave a construction of
such an almost $t$-wise permutation of size $n^{O(t)}$.

\subsection{Paper organization}

We give a precise description of the general framework and our main
theorem in Section~\ref{sec:framework}. We apply it to show the
existence and estimate the number of orthogonal arrays, $t$-designs
and $t$-wise permutations in Section~\ref{sec:applications}. The
proof of our main theorem is given in
Section~\ref{sec:proof_main_theorem}. We summarize and give some
open problems in Section~\ref{sec:summary}.

\section{General framework}
\label{sec:framework}

Let $B$ be a finite set and let $V$ be a linear subspace of
functions $f:B \to \Q$. The goal of this work is to find sufficient
conditions for the existence of a small set $T \subset B$ such that
\begin{equation}\label{eq:prob_def_2}
\frac{1}{|T|}\sum_{t\in T} f(t) = \frac{1}{|B|}\sum_{b\in B}
f(b)\quad\text{for all $f$ in $V$,}
\end{equation}
and moreover to estimate the number of such sets of a given size. We
now describe a list of conditions on $V$ which will be sufficient
for this task. Some of these conditions are easy to verify in
applications, while others require some effort. We stress that in
all of our applications these properties are verified explicitly;
this is in contrast with the fact that we do not know how to find
$T$ explicitly.

\paragraph{Divisibility.} For \eqref{eq:prob_def_2} to hold for a
set $T$ with $|T|=N$ we must have
$$
\sum_{t\in T} f(t) = \frac{N}{|B|}\sum_{b\in B} f(b)\quad\text{for
all $f$ in $V$.}
$$
In particular, we must have that
\begin{equation}\label{eq:div_condition_basis_free}
  \text{there exists a $\gamma\in\Z^B$ such that}\ \sum_{b\in B} \gamma_b f(b) = \frac{N}{|B|}\sum_{b\in B}
f(b)\ \text{for all $f$ in $V$.}
\end{equation}
The set of all integers $N$ satisfying
\eqref{eq:div_condition_basis_free} consists of all integer
multiples of some minimal positive integer $c_1$. To see this,
observe that if $N_1$ and $N_2$ are solutions then their difference
is also a solution. It follows that $|T|$ must be an integer
multiple of $c_1$. This is the divisibility condition that we
require, and we call this $c_1$ the \emph{divisibility constant of
$V$}.

We remark that to check the divisibility condition in practice it
suffices to check \eqref{eq:div_condition_basis_free} for $f$
belonging to some basis of $V$. More explicitly, we make the
following definition.
\begin{dfn}
Let $\phi:B \to \Q^A$ for some finite set $A$. We define $\LL(\phi)$
to be the \emph{lattice spanned by $\phi(b)$, $b \in B$}. I.e.,
$$
\LL(\phi) := \left\{\sum_{b \in B} n_b \cdot \phi(b): n_b \in
\Z\right\} \subset \Q^A.
$$
\end{dfn}
Using this definition, if $\phi:B\to\Q^A$ is such that the vectors
$(\phi_a)$, $a\in A$, form a basis for $V$, then the divisibility
condition is equivalent to having $\frac{N}{|B|}\sum_{b\in
B}\phi(b)\in\LL(\phi)$.

\paragraph{Boundedness.} We make the following definition.
\begin{dfn}
Let $W \subset \Q^B$ be a vector space. For $1\le p\le \infty$, we
say that $W$ has a \emph{$c$-bounded integer basis in $\ell_p$} if
$W$ is spanned by integer vectors whose $\ell_p$ norm is at most
$c$. That is, if
$$
\Span(\{\gamma \in W \cap \Z^B:  \|\gamma\|_p \le c\})=W.
$$
\end{dfn}
The reader should note that the term integer basis is used here with
a different meaning than in Abelian group theory. We will only use
in this paper the norms $\|\gamma\|_1 = \sum_{b \in B} |\gamma_b|$
and $\|\gamma\|_{\infty} = \max_{b \in B} |\gamma_b|$. We denote by
$V^{\perp}$ the orthogonal complement of $V$ in $\Q^B$, that is,
$$
V^{\perp} := \{g \in \Q^B\,:\, \sum_{b \in B} f(b)g(b)=0 \quad
\forall f \in V\}.
$$
We impose the conditions that for some small $c_2$ and $c_3$, $V$
has a $c_2$-bounded integer basis in $\ell_{\infty}$ and $V^{\perp}$
has a $c_3$-bounded integer basis in $\ell_1$.

In our applications, the boundedness condition for $V$ follows
directly from the definition. However, the boundedness condition for
$V^\perp$ is less trivial and requires far more work to check. We
view this condition as an analog of the LDPC (Low Density Parity
Check) condition in coding theory, when viewed over the integers,
and we develop techniques based on coding theory in order to
guarantee it. In particular, we show in
Section~\ref{sec:local_decodability} that the condition is implied
by a certain local decodability property of $V$.

\paragraph{Symmetry.} The next condition relates to the symmetries of the subspace
$V$.
\begin{dfn}
  A \emph{symmetry of $V$} is a permutation $\pi \in S_B$ satisfying
$f \circ \pi\in V$ for all $f \in V$.
\end{dfn}
Equivalently, if $\phi:B\to\Q^A$ is such that the vectors
$(\phi_a)$, $a\in A$, form a basis for $V$ then a permutation
$\pi\in S_B$ is a symmetry of $V$ if and only if there exists an
invertible linear map $\tau:\Q^A \to \Q^A$ such that
$$
\phi(\pi(b)) = \tau(\phi(b))\quad\text{for all $b\in B$}.
$$
It is straightforward to check that the set of symmetries of $V$
forms a subgroup of $S_B$. We impose the condition that this group
acts transitively on $B$. That is, that for any $b_1,b_2 \in B$
there exists a symmetry $\pi$ of $V$ satisfying $\pi(b_1)=b_2$. In
our applications this condition follows from the symmetric nature of
their description.

\paragraph{Constant functions.} Our last condition is required for somewhat technical reasons as
explained in the proof overview section. We require the constant
functions to belong to $V$. We note that this is the case in all of
our applications.

We can now state our main theorem.

\begin{thm}[Main Theorem]\label{thm:main}
There exists a constant $C>0$ such that the following is true. Let
$B$ be a finite set and let $V$ be a linear subspace of functions
$f:B \to \Q$. Assume that the following conditions hold for some
integers $c_1,c_2,c_3 \ge 1$,
\begin{enumerate}
\item Divisibility: $c_1$ is the divisibility constant of $V$.
\item Boundedness of $V$: $V$ has a $c_2$-bounded integer basis in $\ell_{\infty}$.
\item Boundedness of $V^{\perp}$: $V^{\perp}$ has a $c_3$-bounded integer basis in $\ell_{1}$.
\item Symmetry: for any $b_1,b_2 \in B$ there exists a symmetry $\pi$ of $V$ satisfying $\pi(b_1)=b_2$.
\item Constant functions: The constant functions belong to $V$.
\end{enumerate}
If
\begin{equation}\label{eq:N_condition}
  \text{$N$ is an integer multiple of $c_1$ satisfying $\min(N, |B|-N) \ge
C \cdot c_2 c_3^2\dim(V)^6\log(2c_3\dim(V))^6$}
\end{equation}
then there exists a subset $T\subset B$ of size $|T|=N$ satisfying
\begin{equation}\label{eq:prob_def_3}
\frac{1}{|T|}\sum_{t\in T} f(t) = \frac{1}{|B|}\sum_{b\in B}
f(b)\quad\text{for all $f$ in $V$.}
\end{equation}
\end{thm}

A second goal of our work is to count the number of subsets $T$ of a
given size $N$ which satisfy \eqref{eq:prob_def_3}. To this end, we
define a parameter $\rho(V)$ of the vector space $V$ as follows. It
is easiest to define $\rho(V)$ via a choice of basis for $V$ but we
stress that its value is independent of this choice. If
$\phi:B\to\Q^A$ is such that the vectors $(\phi_a)$, $a\in A$, form
a basis for $V$, we define
\begin{equation}\label{eq:rho_formula}
  \rho(V):=\frac{\det(\LL(\phi))}{\sqrt{\det(\phi^t\phi)}},
\end{equation}
where in the numerator we mean the determinant of the lattice
$\LL(\phi)$ (i.e., the volume in the appropriate dimension of a
fundamental parallelogram of the lattice generated by $\{\phi(b)\}$,
$b\in B$) and in the denominator, the root of the determinant of the
$A\times A$ matrix whose $a,a'$ entry is $\sum_{b\in B} \phi(b)_a
\phi(b)_{a'}$ (we denote by $\phi^t$ the transpose of $\phi$). The
definition takes a more symmetric form upon noting that the
denominator is the determinant of the lattice in $\Q^B$ generated by
$\{\phi_a\}$, $a\in A$.

\begin{thm}\label{thm:count_solutions}
There exists a constant $C>0$ such that the following is true. Let
$B$ be a finite set and let $V$ be a linear subspace of functions
$f:B \to \Q$. Assume that the conditions of Theorem~\ref{thm:main}
are satisfied with constants $c_1,c_2,c_3\ge 1$ and that $N$
satisfies \eqref{eq:N_condition}. Then the number of subsets
$T\subset B$ of size $N$ which satisfy \eqref{eq:prob_def_3} equals
\begin{equation*}
  \frac{\rho(V)}{(2\pi
  p(1-p))^{\frac{\dim(V)}{2}}p^N(1-p)^{|B|-N}}(1+\delta)
\end{equation*}
where $p:=\frac{N}{|B|}$ and $|\delta|\le \frac{C\dim(V)^3
(\log(2c_2\dim(V)))^{3/2}}{\sqrt{\min(N,|B|-N)}}$.
\end{thm}
As explained before, the divisibility requirement in
\eqref{eq:N_condition} is a necessary condition for the existence of
a subset $T\subset B$ of size $N$ satisfying \eqref{eq:prob_def_3}.
Our theorem says that when $N$ is not too close to $0$ or $|B|$ this
condition is also sufficient, and gives a rather precise count of
the number of such subsets.

Finally, we remark that our techniques yield a bit more. One can use
them to show the existence, and estimate the number, of subsets $T$
of a given size on which the average of functions $f\in V$ has a
specified (small) displacement from the average over all of $B$.
This extension is described in Sections~\ref{sec:LCLT_with_basis}
and \ref{sec:LCLT_basis_free}.

\section{Applications}
\label{sec:applications}

In this section we apply our main theorem, Theorem~\ref{thm:main},
to prove the existence results for orthogonal arrays and
$t$-designs,
Theorems~\ref{thm:OA_first_version},~\ref{thm:OA_count},~\ref{thm:designs_first_version}
and~\ref{thm:perms_first_version}. It will be useful to introduce
the following notation. For a map $\phi:B\to \Q^A$ and a vector
$\gamma\in\Q^B$ we let
\begin{equation*}
  \phi(\gamma) = \sum_{b\in B} \gamma_b \phi(b) \in \Q^A.
\end{equation*}
We also define
\begin{equation*}
  \|\phi\|_{\infty} := \max_{b\in B, a\in A} |\phi(b)_a|.
\end{equation*}
We start by describing a certain condition which implies the
boundedness condition for $V^\perp$ and which will be useful in our
applications to orthogonal arrays and $t$-designs.

\subsection{Local decodability}\label{sec:local_decodability}
In all of our applications it turns out that the most difficult
condition to verify is that $V^{\perp}$ has a bounded integer basis
in $\ell_1$. This condition can be seen as an analog of the Low
Density Parity Check (LDPC) notion coming from coding theory. We
next introduce another condition which implies that $V^{\perp}$ has
a bounded integer basis in $\ell_1$, but which is sometimes easier
to verify in practice. This condition is motivated by the notion of
locally decodable codes in coding theory. Local decodability of
codes is mainly studied in the context of codes defined over finite
fields, see, e.g.,~\cite{Yekhanin}. Here, we study an analog of
these definitions for codes defined over the rationals.

Formally, we require that for some
bounded integer basis $(\phi_a)$, $a\in A$, of $V$, we may express a
small multiple of the unit vectors (in the basis given by $A$) by
short integer combinations of $\phi(b)$. The condition is also related to
the notion of bi-orthogonal system in functional analysis.

\begin{dfn}[Local decodability]\label{def:local_decodability}
A map $\phi:B \to \Z^A$ is \emph{locally decodable with bound $c$}
if there exists an integer $m \ge 1$ with $|m| \le c$ and a set of
vectors $(\gamma^a)\subset\Z^B$, $a\in A$, satisfying
$\|\gamma^a\|_1 \le c$ and
$$
\phi(\gamma^a) = m \cdot e^a\qquad(a\in A),
$$
where $e^a \in \{0,1\}^A$ is the unit vector with $1$ in coordinate
$a$.
\end{dfn}

\begin{claim}\label{claim:LDC_implies_LDPC}
If $V$ has a basis of integer vectors $(\phi_a)\subset\Z^B$, $a\in
A$, such that $\|\phi\|_{\infty} \le c_2$ and $\phi$ is locally
decodable with bound $c_4$ then $V^{\perp}$ has a $c_3$-bounded
integer basis in $\ell_1$ with $c_3 \le 2 c_2 c_4 |A|$.
\end{claim}

\begin{proof}
Let $m$ and $(\gamma^a)$, $a\in A$, be as in
Definition~\ref{def:local_decodability} for $\phi$. Define the
vectors $(\delta^b)\subset\Z^B$, $b\in B$ by
$$
\delta^b := m \cdot u^b - \sum_{a \in A} \phi(b)_a \cdot \gamma^a,
$$
where $u^b \in \{0,1\}^B$ is the unit vector with $1$ in coordinate
$b$. We claim that the set $\{\delta^b: b \in B\}$ linearly spans
$V^{\perp}$. First, note that $\delta^b \in V^{\perp}$ for all $b
\in B$ since
$$
\phi(\delta^b) = m\cdot\phi(u^b) - \sum_{a\in A} \phi(b)_a\cdot
\phi(\gamma^a) = m \cdot \phi(b) - \sum_{a \in A} \phi(b)_a \cdot (m
\cdot e_a) = 0.
$$
We next argue that the rank of $\{\delta_b: b \in B\}$ is at least
$|B|-|A|$, and hence they must span $V^{\perp}$. To see this, let
$\Psi:\Q^B \to \Q^{|B|-|A|}$ be an arbitrary surjective linear map
which sends $\{\gamma^a: a \in A\}$ to zero. Then $\{\delta_b: b \in
B\}$ are mapped to a basis of $\Q^{|B|-|A|}$ by $\Psi$ and hence
their dimension is at least $|B|-|A|$ (in other words, the
$\{\delta^b\}$ are a linear perturbation of the $B\times B$ identity
matrix by a matrix whose rank is at most $|A|$). The bound on $c_3$
follows since
\begin{equation*}
\|\delta_b\|_1 \le m + \sum_{a \in A} |\phi(b)_a| \|\gamma^a\|_1 \le
c_4 + |A| c_2 c_4 \le 2 c_2 c_4 |A|\qquad(b\in B).\qedhere
\end{equation*}
\end{proof}

\subsection{Orthogonal arrays}
\label{sec:orthogonal_arrays}

We prove Theorems~\ref{thm:OA_first_version} and \ref{thm:OA_count}
in this subsection. We recall the relevant definitions from the
introduction. A subset $T \subset [q]^n$ is an \emph{orthogonal
array of alphabet size $q$, length $n$ and strength $t$} if it
yields all strings of length $t$ with equal frequency when
restricted to any $t$ coordinates.  In other words, for any distinct
indices $i_1,\ldots,i_t \in [n]$ and any (not necessarily distinct)
values $v_1,\ldots,v_t \in [q]$,
$$
\left|\{x \in T: x_{i_1}=v_1,\ldots,x_{i_t}=v_t\}\right| = q^{-t} |T|.
$$

Orthogonal arrays fit into our general framework as follows. We take $B:=[q]^n$ and $V$ to be the space spanned by all functions of the form
\begin{equation*}
  f_{(I,v)}(x_1,\ldots, x_n) = \begin{cases}
    1& x_i=v_i\text{ for all $i\in I$}\\0&\text{Otherwise}
  \end{cases},
\end{equation*}
with $I\subset[n]$ a subset of size $t$ and $v\in[q]^I$.  With this
choice, a subset $T\subset B$ satisfying \eqref{eq:prob_def} is
precisely an orthogonal array of alphabet size $q$, length $n$ and
strength $t$.

We shall now verify the conditions of Theorem~\ref{thm:main} for
$V$. We note that the sum of all the above $f_{(I,v)}$ is a constant
function, thus verifying the constant functions condition. We
continue with the symmetry condition. Fix $x \in [q]^n$ and consider
the permutation $\pi_x \in S_B$ given by $\pi_x(b)=b+x \pmod{q}$,
where we apply the modulo $q$ coordinate-wise, and with the
convention that it maps $\Z$ to $[q]$. We will show that each
$\pi_x$ is a symmetry of $V$, which will establish the condition
since the group $\{\pi_x: x \in [q]^n\}$ acts transitively on $B$.
It suffices to show that for $(I,v)$ of the above type, we have
$f_{(I,v)}\circ\pi_x \in V$. Indeed,
$$
(f_{(I,v)}\circ\pi_x)(b) = f_{(I,v)}(b+x \textrm{ (mod }{q})) =
f_{(I,v')}(b)\in V,
$$
where $v'_i = v_i - x_i \pmod{q}$ for $i\in I$.

To verify the remaining conditions in Theorem~\ref{thm:main} we
choose a convenient basis for $V$. The above set of functions
$\{f_{I,v}\}$ is linearly dependent in general. Let
\begin{equation*}
A:=\{(I,v): |I| \le t, v \in [q-1]^{I}\},
\end{equation*}
and set $\phi_a=f_{(I,v)}$ for $a=(I,v) \in A$. Here, by
$f_{(\emptyset,\emptyset)}$ we mean the constant one function. We
will show that the $(\phi_a)$, $a\in A$, form a basis for $V$, that
the lattice $\LL(\phi)$ which they generate equals $\Z^A$ and that
they are locally decodable. The remaining conditions of
Theorem~\ref{thm:main} will follow easily from these properties.

\begin{claim} \label{cl:OA_A_spans_V} The span of the functions
$\{\phi_a\}_{a \in A}$ is $V$.
\end{claim}

\begin{proof}
Clearly $\phi_a \in V$ for all $a \in A$. To see that the
$\{\phi_a\}_{a \in A}$ also span $V$, it suffices to show that any
$f_{(I,v)}$ with $|I| \le t$ and $v \in [q]^I$ is in the span of
$\{\phi_a\}_{a \in A}$. We do this by induction on the number of
elements in $v$ which are equal to $q$. Let $(I,v)$ have $|I|\le t$
and $v\in[q]^I$. First, if $v\in[q-1]^I$ then $(I,v)\in A$ by
definition. Now suppose $v\in [q]^I\setminus [q-1]^I$ and let
$i_0\in I$ be such that $v_{i_0}=q$. For $j\in[q-1]$ define $v^j$ by
$v^j_i=v_i$, $i\in I\setminus\{i_0\}$, and $v^j_{i_0}=j$. Define
$v'$ to be the restriction of $v$ to $I\setminus\{i_0\}$. Then
$$
f_{(I,v)} = f_{(I\setminus\{i_0\}, v')} - \sum_{j=1}^{q-1}
f_{(I,v^j)}
$$
and, by induction, the right hand side belongs to the linear span of
$\{\phi_a\}_{a \in A}$.
\end{proof}

\begin{claim}\label{cl:OA_local_decodability}
  The map $\phi$ is locally decodable with bound $2^t$ and $m=1$.
  Consequently, the $(\phi_a)$, $a\in A$, form a basis for $V$ and
  $\LL(\phi)=\Z^A$.
\end{claim}
\begin{proof}
For $J\subset I\subset[n]$ and $v\in[q]^I$, we write $v|_J$ for the
restriction of $v$ to $J$. Define a partial order on $A$ by letting
$a'\le a$, for $a=(I,v)$ and $a'=(I',v')$, if $I'\subset I$ and
$v'=v|_{I'}$. For each $a=(I,v)\in A$, define an element
$b^a\in[q]^n$ by
\begin{equation*}
b^{(a)}_i := \begin{cases}
v_i & \textrm{if } i \in I\\
q & \textrm{if } i \notin I
\end{cases}.
\end{equation*}
The definition of $\phi$ implies that
\begin{equation*}
  \phi_{a'}(b^a) = 1_{(a'\le a)}\quad(a,a'\in A).
\end{equation*}
Define for each $a=(I,v) \in A$ the vector $\gamma^{a} \in \Z^B$ by
\begin{align*}
\gamma^a_b:=\begin{cases} (-1)^{|I|-|J|}&b=b^{(J,v_J)}\text{ for
some
}J\subset I\\
0&\text{otherwise}
\end{cases}.
\end{align*}
We have $\|\gamma^a\|_1=2^{|I|}\le 2^t$. We will show that
$\phi(\gamma^a)=e^a$ for each $a\in A$, thereby establishing the
local decodability claim. By the above, if $a=(I,v)$ and
$a'=(I',v')$ then
\begin{align*}
  \phi(\gamma^a)_{a'} &= \sum_{J\subset I}
  (-1)^{|I|-|J|}\phi_{a'}(b^{(J,v|_J)}) = \sum_{J\subset I}
  (-1)^{|I|-|J|}1_{((I',v')\le (J,v|_J))} =\\
  &= 1_{(v'=v|_{I'})}
  \sum_{I'\subset J\subset I} (-1)^{|I|-|J|}=1_{(v'=v|_{I'})}\sum_{j=0}^{|I|-|I'|} (-1)^{j} {|I|-|I'| \choose
j}=\delta_{a,a'}
\end{align*}
as we wanted to show. Local decodability implies that the
$(\phi_a)$, $a\in A$, have full rank, and together with
Claim~\ref{cl:OA_A_spans_V} we deduce that they form a basis for
$V$. Additionally, the fact that $\phi(\gamma^a)=e^a$ for every
$a\in A$ implies that $\LL(\phi)=\Z^A$.
\end{proof}
\begin{claim}\label{cl:OA_div_constant}
  The divisibility constant of $V$ equals $q^t$.
\end{claim}
\begin{proof}
  It suffices to show that $q^t$ is the smallest positive integer $N$ for which
  \begin{equation*}
    \frac{N}{|B|}\sum_{b\in B} \phi(b) \in \LL(\phi)
  \end{equation*}
  Indeed, since $\LL(\phi)=\Z^A$ by
  Claim~\ref{cl:OA_local_decodability} this follows by noting that
  \begin{equation*}
    \frac{1}{|B|} \sum_{b\in B}\phi(b)_{(I,v)} = \frac{1}{q^n}|\{x\in [q]^n\,:\,x_i=v_i\; \forall i \in I\}| =
    q^{-|I|}.\qedhere
  \end{equation*}
\end{proof}

We are now in place to apply Theorem~\ref{thm:main}. The
divisibility constant of $V$ is $c_1=q^t$. The $\{\phi_a\}$, $a\in
A$, are a $1$-bounded integer basis for $V$ giving $c_2=1$. We have
\begin{equation}\label{eq:OA_dim_V}
  \dim(V) = |A| = \sum_{i=0}^t \binom{n}{i}(q-1)^i\le
  \binom{n}{t}q^t \le \left(\frac{eqn}{t}\right)^t
\end{equation}
with the next to last inequality following since $V$ was defined as
the span of $f_{(I,v)}$ with $|I|=t$ and $v\in [q]^I$. Since local
decodability holds with bound $c_4=2^t$, we deduce from
Claim~\ref{claim:LDC_implies_LDPC} that $V^\perp$ has a
$c_3$-bounded integer basis in $\ell_1$ with $c_3\le 2c_2 c_4|A|\le
2\big(\frac{2eqn}{t}\big)^t$. Hence we establish the existence of an
orthogonal array of alphabet size $q$, length $n$, strength $t$ and
size $|T| \le \big(\frac{cqn}{t}\big)^{ct}$ for some universal
constant $c>0$.

Lastly, we aim to use Theorem~\ref{thm:count_solutions} to count the
number of orthogonal arrays of a given size. To this end we need
only calculate $\rho(V)$. Observing that for our choice of $\phi$ we
have $\det(\LL(\phi))=1$ by Claim~\ref{cl:OA_local_decodability}, we
may apply formula~\eqref{eq:rho_formula} (with our choice of $\phi$)
to obtain
\begin{equation*}
  \rho(V)=\frac{\det(\LL(\phi))}{\sqrt{\det(\phi^t\phi)}} = \frac{1}{\sqrt{\det(\phi^t\phi)}}.
\end{equation*}
\begin{claim}\label{cl:OA_det_calc}
  $\det(\phi^t\phi)=q^{n\binom{n-1}{t}(q-1)^t}$.
\end{claim}
It follows from the claim that
\begin{equation*}
  \rho(V) = q^{-n\binom{n-1}{t}(q-1)^t/2}.
\end{equation*}
Together with the calculation of $\dim(V)$ in \eqref{eq:OA_dim_V},
Theorem~\ref{thm:OA_count} now follows from
Theorem~\ref{thm:count_solutions}.

\begin{proof}[Proof of Claim~\ref{cl:OA_det_calc}]
It will be useful to let $n$ and $t$ vary in the proof of the claim.
Hence we shall write, for $n\ge 1$ and $0\le t\le n$, $A(n,t)$ for
$A$ and $\phi(n,t)$ for $\phi$. Denote
$R(n,t):=\phi(n,t)^t\phi(n,t)$ (where $\phi(n,t)^t$ denotes the
transpose of $\phi(n,t)$). Then $R(n,t)$ is an $A(n,t)\times A(n,t)$
matrix satisfying
\begin{equation}\label{eq:OA_R_def}
  R(n,t)_{(I,v), (I',v')} = |\{(x_1,\ldots, x_n)\in [q]^n\,:\,
  x_i=v_i\,\forall i\in I\text{ and }x_{i'}=v_{i'}\,\forall i'\in
  I'\}|.
\end{equation}
We also let
\begin{equation}\label{eq:OA_A_size}
  d_{n,t} := |A(n,t)| = \sum_{i=0}^t \binom{n}{i}(q-1)^i.
\end{equation}
Define $a_{n,t}:=\log_q(\det(R(n,t)))$ so that the claim states
\begin{equation}\label{eq:OA_a_n_t_def}
  a_{n,t} = n\binom{n-1}{t}(q-1)^t.
\end{equation}
We first establish this fact when $t=0$ or $n=t$. Indeed, if $t=0$
we have $A(n,t) = \{(\emptyset,\emptyset)\}$ and
$R(n,t)_{(\emptyset,\emptyset),(\emptyset,\emptyset)}=q^{n}$,
proving that $a_{n,0}=n$. If $n=t$, $|A(n,t)|=q^n$ and hence
$\phi(n,t)$ is a \emph{square} matrix with the property that the
lattice spanned by its rows, by
Claim~\ref{cl:OA_local_decodability}, is $\Z^{A(n,t)}$. Thus
$\det(\phi(n,t))=1$ and hence $\det(R(n,t))=1=q^0$, verifying
\eqref{eq:OA_a_n_t_def} in this case as well.

In the rest of the proof we will show that for any $n>t>0$ we have
\begin{equation}\label{eq:OA_det_recursion}
  a_{n,t} = a_{n-1,t} + d_{n-1,t} + (q-1) a_{n-1,t-1} - d_{n-1,t-1}.
\end{equation}
Noting that $d_{n-1,t} - d_{n-1,t-1} = \binom{n-1}{t}(q-1)^t$ by
\eqref{eq:OA_A_size}, the claim follows upon verifying that the
$(a_{n,t})$ defined by \eqref{eq:OA_a_n_t_def} satisfy this
recursion.

To prove \eqref{eq:OA_det_recursion}, fix $n>t>0$ and partition
$A(n,t)$ as follows:
\begin{align*}
  A(n,t)^0&:=\{(I,v)\in A(n,t)\,:\, n\notin I\},\\
  A(n,t)^j&:=\{(I,v)\in A(n,t)\,:\, n\in I\text{ and
  }v_n=j\},\quad j\in[q-1].
\end{align*}
Observe that $A(n-1,t-1)\subset A(n-1,t) = A(n,t)^0$. Denote by
$R(n,t)^{i,j}$ the sub-matrix of $R(n,t)$ whose rows are indexed by
$A(n,t)^i$ and whose columns are indexed by $A(n,t)^j$. By
\eqref{eq:OA_R_def} we have
\begin{equation}\label{eq:R_diag_blocks}
  R(n,t)^{i,j} = \begin{cases}
    q R(n-1,t)&i=j=0\\
    0&i,j\in[q-1], i\neq j\\
    R(n-1,t-1)&i,j\in [q-1], i=j
  \end{cases},
\end{equation}
where in the second case we mean the $0$ matrix, and in the third
case we identified $a=(I,v)\in A(n,t)^j$ with $a'=(I',v')\in
A(n-1,t-1)$ defined by letting $I'=I\setminus\{n\}$ and $v'_i=v_i$
for $i\in I'$. Summarizing these equalities we have
\begin{equation*}
  R(n,t) = \left(\begin{array}{ccccc}
                     q R(n-1, t) & R(n,t)^{0,1} & R(n,t)^{0,2} & \cdots & R(n,t)^{0,q-1} \\
                     R(n,t)^{1,0} & R' & 0 & \cdots & 0 \\
                     R(n,t)^{2,0} & 0 & R' & \cdots & 0 \\
                     \vdots & \vdots & \vdots & \ddots & \vdots \\
                     R(n,t)^{q-1,0} & 0 & 0 & \cdots & R' \\
                   \end{array}\right),
\end{equation*}
where the $i,j$ cell in the displayed matrix corresponds to the
sub-matrix indexed by $A(n,t)^i$ and $A(n,t)^j$, and where we
abbreviated $R':=R(n-1,t-1)$ for display purposes. In addition,
\eqref{eq:OA_R_def} implies that if $a_1\in A(n,t)^{j_1}$ and
$a_2\in A(n,t)^{j_2}$ then
\begin{equation}\label{eq:R_off_diag_blocks}
  R(n,t)_{a_1,a_2} = \begin{cases}
    \frac{1}{q}R(n,t)_{a_1,a_2'}&j_1=0, j_2\neq 0\\
    R(n,t)_{a_1,a_2'}&j_1, j_2\neq 0, j_1=j_2
  \end{cases}.
\end{equation}
Thus, if we subtract from each of the columns indexed by $a_2\in
A(n,t)\setminus A(n,t)^0$ the corresponding column indexed by
$a_2'\in A(n,t)^0$ times $\frac{1}{q}$ we obtain the matrix
$\tilde{R}(n,t)$ satisfying
\begin{equation*}
  \tilde{R}(n,t) = \left(\begin{array}{ccccc}
                     q R(n-1, t) & 0 & 0 & \cdots & 0 \\
                     R(n,t)^{1,0} & (1-\frac{1}{q})R' & -\frac{1}{q}R' & \cdots & -\frac{1}{q}R' \\
                     R(n,t)^{2,0} & -\frac{1}{q}R' & (1-\frac{1}{q})R' & \cdots & -\frac{1}{q}R' \\
                     \vdots & \vdots & \vdots & \ddots & \vdots \\
                     R(n,t)^{q-1,0} & -\frac{1}{q}R' & -\frac{1}{q}R' & \cdots & (1-\frac{1}{q})R' \\
                   \end{array}\right),
\end{equation*}
Denoting by $I$ the identity matrix and by ${\bf 1}$ the square
matrix containing all ones, both of dimension $q-1$, it follows that
\begin{align*}
  \det(R(n,t)) &= \det(\tilde{R}(n,t)) = \det(q
  R(n-1,t))\det\left(\left(I-\frac{1}{q}{\bf 1}\right)\otimes R(n-1,t-1)\right) =\\
  &= q^{a_{n-1,t}+d_{n-1,t}+(q-1)a_{n-1,t-1}}\det\left(I-\frac{1}{q}{\bf 1}\right)^{d_{n-1,t-1}}
  = q^{a_{n-1, t} + d_{n-1,t} + (q-1)a_{n-1,t-1} - d_{n-1,t-1}},
\end{align*}
where we used that if $A$ is an $m\times m$ matrix and $B$ is a
$k\times k$ matrix then $\det(A\otimes B)=\det(A)^k\det(B)^m$, and
where we calculated $\det(I-\frac{1}{q}{\bf 1})=\frac{1}{q}$ since
the only non-zero eigenvalue of ${\bf 1}$ is $q-1$. This establishes
the recursion \eqref{eq:OA_det_recursion} and finishes the proof of
the claim.
\end{proof}

\subsection{Designs}
\label{sec:designsproof}

We prove Theorems~\ref{thm:designs_first_version}
and~\ref{thm:counting_designs} in this subsection. We recall the
relevant definitions from the introduction. A (simple)
\emph{$t$-$(v,k,\lambda)$ design}, or $t$-design for short, is a
family of distinct subsets of $[v]$, where each set is of size $k$,
such that each $t$ elements belong to exactly $\lambda$ sets. In
other words, denoting by $\choosesq{v}{k}$ the family of all subsets
of $[v]$ of size $k$, a set $T \subset \choosesq{v}{k}$ is a
$t$-design if for any distinct elements $i_1,\ldots,i_t \in [v]$,
\begin{equation}\label{eq:design_def}
\left|\{s \in T: i_1,\ldots,i_t \in s\}\right|
    = \frac{\binom{k}{t}}{\binom{v}{t}} |T|=\lambda.
\end{equation}

Our general framework includes $t$-designs as follows. We take $B$
to be $\choosesq{v}{k}$ and $V$ to be the space spanned by all
functions of the form
\begin{equation*}
  f_{a}(b) = \begin{cases}
    1& a\subset b\\0&\text{Otherwise}
  \end{cases},
\end{equation*}
with $a\in \choosesq{v}{t}$. With this choice, a subset $T\subset B$
satisfying \eqref{eq:prob_def} is precisely a simple
$t$-$(v,k,\lambda)$ design. We choose $A=\choosesq{v}{t}$, set
$\phi_a=f_a$ and define a map $\phi:B \to \Z^A$ by
$\phi(b)_a=\phi_a(b)$.

Fix $v\ge 1$ and $1\le t\le k\le v$. We assume without loss of
generality that $k\le v-t$ since if $k>v-t$ we have
$|B|=\binom{v}{k}\le\binom{v}{t}$ and hence
Theorem~\ref{thm:designs_first_version} holds trivially and
Theorem~\ref{thm:counting_designs} holds vacuously (by taking the
complete design in Theorem~\ref{thm:designs_first_version}, and by
noting that necessarily $\binom{v}{k} - N\le \binom{v}{t}$ in
Theorem~\ref{thm:counting_designs}). This assumption will be needed
shortly to show that $\{\phi_a\}$, ${a \in A}$, form a basis for $V$
(that is, that they are linearly independent).

We shall now verify the conditions of Theorem~\ref{thm:main} for
$V$. First, the boundedness condition for $V$ trivially holds with
constant $c_2=1$ by our choice of $\phi$. Second, we observe that
$\sum_{a\in A} \phi_a$ is the vector with all coordinates equal to
$\binom{k}{t}$. Hence the constant functions assumption is
satisfied. Third, to establish the symmetry condition let $\pi \in
S_{[v]}$ be a permutation on $[v]$. $\pi$ acts in a natural way on
$B$ (by permuting $k$-sets) and on $A$ (by permuting $t$-sets). We
have that
$$
(\phi_a\circ\pi)(b)=\phi_a(\pi(b)) = 1_{a \subset \pi(b)} =
1_{\pi^{-1}(a) \subset b}=\phi_{\pi^{-1}(a)}(b),
$$
and in particular $\phi_a\circ\pi \in V$ for all $a \in A$. The
action of $S_{[v]}$ on $B$ is transitive, from which the symmetry
condition follows. We continue to find the divisibility constant of
$V$. We need the following result of Wilson \cite{Wilson73} and
Graver and Jurkat \cite{GraverJurkat}.
\begin{thm}
  The vector in $\Q^A$ all of whose coordinates equal $\lambda$
belongs to $\LL(\phi)$ if and only if
\begin{equation*}
  \binom{v-s}{t-s}\lambda \equiv 0\,\,\bmod{\binom{k-s}{t-s}}\quad \text{for
  all $0\le s\le t$}.
\end{equation*}
\end{thm}
Define $c_1\ge 1$ to be the minimal integer such that
\begin{equation}\label{eq:design_c1_def}
\binom{k}{s}c_1 \equiv 0\,\,\bmod{\binom{v}{s}}\quad \text{for all
$1\le s\le t$}.
\end{equation}
We claim that $c_1$ is the divisibility constant of $V$. Indeed,
since all the coordinates of $\frac{N}{|B|}\sum_{b\in B}\phi(b)$
equal $N\binom{v-t}{k-t}/\binom{v}{k}=N\binom{k}{t}/\binom{v}{t}$ we
see that this vector belongs to $\LL(\phi)$ if and only if
\begin{equation*}
  \frac{\binom{v-s}{t-s}\binom{k}{t}}{\binom{k-s}{t-s}\binom{v}{t}}N
  = \frac{\binom{k}{s}}{\binom{v}{s}}N\in\Z\quad \text{ for
  all $0\le s\le t$}.
\end{equation*}
The case $s=0$ simply means that $N\in \Z$, thus a comparison with
\eqref{eq:design_c1_def} verifies that $c_1$ is the divisibility
constant of $V$.

It is useful to have a simple upper bound for $c_1$. Define
\begin{equation}
\lcm(t) := \lcm\left\{\binom{t}{s}\ :\ 0\le s\le t\right\}.
\end{equation}
Observing that
\begin{equation*}
  \frac{\binom{k}{s}}{\binom{v}{s}}\cdot\binom{v}{t}\binom{t}{s} =
  \binom{v-s}{v-t}\binom{k}{s}\in\Z\quad \text{for all
$1\le s\le t$}
\end{equation*}
we deduce that $c_1\le \binom{v}{t}\lcm(t)$. We note
that~\cite{Farhi09} shows that $\log(\lcm(t))$ is asymptotic to $t$
and mentions effective bounds for it. The next claim provides a
simple self-contained proof of a weaker bound which suffices for our
needs. It follows from the claim that $c_1\le \binom{v}{t}4^t\le (4e
v/t)^t$.

\begin{claim}\label{claim:lcm}
$\lcm(t) \le 4^{t}$ for $t \ge 1$.
\end{claim}
\begin{proof}
Assume by induction that the claim holds up to $t$ (checking also
that it holds for $t=1$) and let us prove
  it for $t$. For each $0\le s\le \lfloor\frac{t}{2}\rfloor$ we have
  \begin{equation*}
    \binom{t}{s} = \frac{t!}{s!(t-s)!} = \frac{t!}{\lceil
    \frac{t}{2}\rceil! s!} \frac{\lceil
    \frac{t}{2}\rceil!}{(t-s)!} = \frac{t!}{\lceil
    \frac{t}{2}\rceil! s!} \prod_{i=\lceil\frac{t}{2}\rceil +1}^{t-s} \frac{1}{i}.
  \end{equation*}
  Since the product of every $m$ consecutive integers is divisible
  by $m!$ (since $\binom{a}{m}$ is an integer for every $a$), using the symmetry of the binomial
  coefficients and the induction hypothesis we have
  \begin{align*}
    \lcm\left(\left\{\binom{t}{s}\ :\ 0\le s\le t\right\}\right) &= \lcm\left(\left\{\frac{t!}{\lceil
    \frac{t}{2}\rceil! s!} \prod_{i=\lceil\frac{t}{2}\rceil +1}^{t-s} \frac{1}{i}\ :\ 0\le s\le
    \left\lfloor\frac{t}{2}\right\rfloor\right\}\right) \le\\
    &\le \lcm\left(\left\{\frac{t!}{\lceil
    \frac{t}{2}\rceil! s!(\lfloor\frac{t}{2}\rfloor - s)!}\ :\ 0\le s\le
    \left\lfloor\frac{t}{2}\right\rfloor\right\}\right) =\\
    &= \binom{t}{\lfloor\frac{t}{2}\rfloor}\lcm\left(\left\{\binom{\lfloor\frac{t}{2}\rfloor}{s}\
:\ 0\le s\le
    \left\lfloor\frac{t}{2}\right\rfloor\right\}\right)\le 2^t
    4^{\lfloor\frac{t}{2}\rfloor}\le 4^t.\qedhere
  \end{align*}
\end{proof}

Thus, to verify the conditions of Theorem~\ref{thm:main}, it remains
only to verify the boundedness assumption for $V^\perp$. We do so by
applying the local decodability claim,
Claim~\ref{claim:LDC_implies_LDPC}, to the map $\phi$. Let $a \in
A=\choosesq{v}{t}$. Let $u \in \choosesq{v}{k+t}$ be a any set of
size $k+t$ such that $a \subset u$ (here we use our assumption that
$k\le v-t$). We denote by $\choosesq{u}{k} \subset B$ the family of
subsets of $u$ of size $k$. Define $\gamma_{a,u} \in \Z^B$ as
$$
\gamma_{a,u} := \sum_{s=0}^t \sum_{b \in \choosesq{u}{k}\,:\, |a
\cap b|=s} (-1)^{t-s} \frac{s! (k-s-1)!}{(k-t-1)!} \cdot u_b,
$$
where $u_b \in \{0,1\}^B$ is the unit vector with $1$ in coordinate
$b$.

\begin{claim}\label{claim:designs_gamma_vecs}
$\phi(\gamma_{a,u})=\frac{k!}{(k-t)!} \cdot e_a$ for all $a\in A$.
\end{claim}

Note that the claim implies, in particular, that the $(\phi_a)$,
$a\in A$, are independent and thus
\begin{equation}\label{eq:design_dim_V}
  \dim(V) = |A| = \binom{v}{t}.
\end{equation}
The next technical claim is used in the proof of
Claim~\ref{claim:designs_gamma_vecs}. We set $\binom{n}{m}=0$
whenever $n<m$.

\begin{claim}\label{cl:binomial_sum_designs}
Let $a>b \ge 0$ and $c \ge 0$. Then
$$
\sum_{i=0}^a (-1)^{i} \binom{a}{i} \binom{c+i}{b}=0.
$$
\end{claim}

\begin{proof}
Let $f(a,b,c) = \sum_{i=0}^a (-1)^i \binom{a}{i} \binom{c+i}{b}$. If
$b,c>0$ we have $\binom{c+i}{b}=\binom{c-1+i}{b}+\binom{c-1+i}{b-1}$
and hence $f(a,b,c) = f(a,b,c-1)+f(a,b-1,c-1)$.  So, it is enough to
verify the claim whenever $b=0$ or $c=0$. If $b=0$ then
$f(a,0,c)=\sum_{i=0}^a (-1)^i \binom{a}{i}=0$ since $a \ge 1$. If
$c=0$ then $f(a,b,0)=\sum_{i=b}^a (-1)^i \binom{a}{i} \binom{i}{b} =
\binom{a}{b} \sum_{i=b}^a (-1)^i \binom{a-b}{i-b}=0$.
\end{proof}

\begin{proof}[Proof of Claim~\ref{claim:designs_gamma_vecs}]
It is clear from the definition that $\phi(\gamma_{a,u})_{a'}=0$ if
$a' \not \subset u$. So, we restrict our attention to $a' \subset
u$. For $a'=a$ the contribution is only from sets with $s=|a \cap
b|=t$, of which there are ${k \choose k-t}$, and hence
$$
\phi(\gamma_{a,u})_a = {k \choose k-t} t! = \frac{k!}{(k-t)!}.
$$
We now need to verify that $\phi(\gamma_{a,u})_{a'}=0$ for all $a'
\subset u, a' \ne a$. Let us denote $\ell=|a \cap a'|$ where $0 \le
\ell < t$. The contribution to $\phi(\gamma_{a,u})_{a'}$ comes only
from sets $b$ for which $a' \subset b$. The number of these sets
with $|a \cap b|=s$ is ${t-\ell \choose s-\ell}{k-t+\ell \choose
k-t-s+\ell}={t-\ell \choose t-s}{k-t+\ell \choose s}$. Note,
moreover, that $s \ge \ell$. We have
\begin{align*}
\phi(\gamma_{a,u})_{a'}
&= \sum_{s=\ell}^{t} {t-\ell \choose t-s}{k-t+\ell \choose s} \cdot (-1)^{t-s} \frac{s! (k-s-1)!}{(k-t-1)!}\\
&= \frac{(t-\ell-1)!(k-t+\ell)!}{(k-t-1)!}\sum_{s=\ell}^t
(-1)^{t-s}\binom{t-\ell}{t-s}\binom{k-s-1}{t-\ell-1}\\
&= \frac{(t-\ell-1)!(k-t+\ell)!}{(k-t-1)!}\sum_{i=0}^{t-\ell} (-1)^i
\binom{t-\ell}{i}\binom{k-t-1+i}{t-\ell-1}.
\end{align*}
Recalling that $t-\ell>0$ we now apply
Claim~\ref{cl:binomial_sum_designs} with
$a=t-\ell,b=t-\ell-1,c=k-t-1$ and conclude that
$\phi(\gamma_{a,u})_{a'}=0$.
\end{proof}

In order to obtain tight bounds, we will divide $\gamma_{a,u}$ by a
factor common to all the coefficients appearing in it. Note that
  $$
    \frac{s! (k-s-1)!}{(k-t-1)!} =
    \binom{k-s-1}{k-t-1}\binom{t}{s}^{-1} t!.
  $$
and hence
\begin{equation}
\gamma'_{a,u} := \frac{\lcm(t)}{t!} \cdot \gamma_{a,u} \in\Z^B
\end{equation}
We continue to show that $\phi$ is locally decodable. We have
$$
\phi(\gamma'_{a,u}) = {k \choose t} \lcm(t) \cdot e_a
$$
and we recall that ${k \choose t} \lcm(t)\le {k \choose t}4^t$ by
Claim~\ref{claim:lcm}. To bound $\|\gamma'_{a,u}\|_1$ observe that
the number of $b \in \choosesq{u}{k}$ for which $|a \cap b|=s$ is
${t \choose s}{k \choose k-s}={t \choose s}{k \choose s}$; and
$\frac{s! (k-s-1)!}{t! (k-t-1)!} = {k-1 \choose t}/{k-1 \choose s}$.
Hence
\begin{align*}
\|\gamma'_{a,u}\|_1 &= \lcm(t) \sum_{s=0}^t {t \choose s}{k \choose
s}\frac{{k-1 \choose t}}{{k-1 \choose s}} \le 4^t\frac{k}{k-t}{k-1
\choose t} \sum_{s=0}^t {t \choose s} = 8^t {k \choose t}
\end{align*}
implying that $\phi$ is locally decodable with bound $c_4 = 8^t {k
\choose t} \le (8e \cdot k/t)^t$. Finally, since $|A|= {v \choose t}
\le (ev/t)^t$ we obtain from Claim~\ref{claim:LDC_implies_LDPC} that
$V^\perp$ has a $c_3$-bounded integer basis in $\ell_1$ with $c_3
\le 2 c_2 c_4 |A| \le (4e v/t)^{2t}$.

We have verified the conditions of Theorem~\ref{thm:main} with
$|A|\le (ev/t)^t, c_1\le (4e v/t)^t,c_2=1, c_3 \le (4e v/t)^{2t}$
and thus we establish the existence of a simple $t$-$(v,k,\lambda)$
design of size $|T| \le (cv/t)^{ct}$ for some universal constant
$c>0$, proving Theorem~\ref{thm:designs_first_version}.

We turn to estimate the number of designs of a given size via
Theorem~\ref{thm:count_solutions}. To this end we consider $\phi$ as
a $B\times A$ matrix and need to calculate the parameter $\rho(V)$
from \eqref{eq:rho_formula}. We rely on a theorem of
Wilson~\cite{Wilson90} giving a diagonal form of $\phi^t$ and on a
theorem of Bapat~\cite{Bapat} calculating the eigenvalues of
$\phi^t\phi$.
\begin{thm}\label{thm:Wilson_diagonal_form}\cite[Theorem 2]{Wilson90}
There exist a $\binom{v}{t}\times\binom{v}{t}$ matrix $E$ and a
$\binom{v}{k}\times\binom{v}{k}$ matrix $F$, both with integer
entries and satisfying $|\det(E)|=|\det(F)|=1$, such that
$M:=E\phi^t F$ has $M_{ij}=0$ if $i\neq j$ and has diagonal entries
$\binom{k-s}{t-s}$ with multiplicity $\binom{v}{s}-\binom{v}{s-1}$
for $0\le s\le t$ (with $\binom{v}{-1}:=0$).
\end{thm}

\begin{thm}\label{thm:design_eigenvalues}\cite[Theorem 4]{Bapat}
The eigenvalues of $\phi^t\phi$ are
$\binom{k-s}{t-s}\binom{v-t-s}{k-t}$ with multiplicity
$\binom{v}{s}-\binom{v}{s-1}$ for $0\le s\le t$ (with
$\binom{v}{-1}:=0$).
\end{thm}

The two theorems immediately imply that
\begin{equation*}
  \rho(V) = \frac{\det(\LL(\phi))}{\sqrt{\det(\phi^t\phi)}} =
  \frac{\prod_{s=0}^t
  \binom{k-s}{t-s}^{\binom{v}{s}-\binom{v}{s-1}}}{\prod_{s=0}^t\left[\binom{k-s}{t-s}\binom{v-t-s}{k-t}\right]^{\frac{1}{2}\left(\binom{v}{s}-\binom{v}{s-1}\right)}}
  = \prod_{s=0}^t
  \left[\frac{\binom{k-s}{t-s}}{\binom{v-t-s}{k-t}}\right]^{\frac{1}{2}\left(\binom{v}{s}-\binom{v}{s-1}\right)}.
\end{equation*}
Theorem~\ref{thm:counting_designs} now follows by an application of
Theorem~\ref{thm:count_solutions} (recalling that
$\dim(V)=\binom{v}{t}$ by \eqref{eq:design_dim_V}).

We remark briefly on a possible approach to proving
Theorems~\ref{thm:Wilson_diagonal_form} and
\ref{thm:design_eigenvalues} via representation theory (see
Section~\ref{sec:defs} for some background), though we neither
require nor develop this approach here. One may naturally identify
the set $\choosesq{v}{k}$ with tabloids of shape $(v-k,k)$ as the
numbers appearing in the second row of the tabloid. Thus, $\R^B$ may
be identified with the Young module $U_{(v-k,k)}$. Similarly, $\R^A$
may be identified with $U_{(v-t,t)}$. For any $m$, the decomposition
of $U_{(v-m,m)}$ into irreducible representations is $U_{(v-m,m)} =
\oplus_{s=0}^m V_{(v-s,s)}$. Since $\phi$ intertwines the action of
$S_n$ on $U_{(v-k,k)}$ and $U_{(v-t,t)}$, Schur's lemma implies that
$\phi$ and $\phi^t$ are diagonal in the basis of irreducible
representations for $U_{(v-k,k)}$ and $U_{(v-t,t)}$, and act as
scalars from $V_{(v-s,s)}$ to itself, $0\le s\le t$. Finally, for
each $s$, the scalars appearing in these actions can be determined
by considering the action on some particular vector in
$V_{(v-s,s)}$. Choosing bases appropriately one obtains that the
scalar for the $\phi$ action is $\binom{v-t-s}{k-t}$ and the scalar
for the $\phi^t$ action is $\binom{k-s}{t-s}$, both with
multiplicity $\dim(V_{(v-s,s)})=\binom{v}{s}-\binom{v}{s-1}$.

\subsection{$t$-wise permutations}
We prove Theorem~\ref{thm:perms_first_version} in this subsection.
We recall the relevant definitions from the introduction. A family of permutations $T \subset S_n$ is called a \emph{$t$-wise permutation} if its action on any $t$-tuple of elements is uniform.
In other words, for any distinct elements $i_1,\ldots,i_t \in [n]$ and
distinct elements $j_1,\ldots,j_t \in [n]$,
\begin{equation*}
    \left| \{\pi \in T: \pi(i_1)=j_1,\ldots,\pi(i_t)=j_t\} \right| =
    \frac{1}{n(n-1)\cdots(n-t+1)} |T|.
\end{equation*}

Our general framework includes $t$-wise permutations as follows. We first set notations. Let $[n]_t:=\{(i_1,\ldots,i_t): i_1,\ldots,i_t \in [n] \textrm{ distinct}\}$ denote the family of $t$-tuples of distinct elements. For $I=(i_1,\ldots,i_t) \in [n]_t$ and $\pi \in S_n$ define $\pi(I):=(\pi(i_1),\ldots,\pi(i_t)) \in [n]_t$. We take
$B=S_n$ and $W$ to be the space spanned by all functions of the form
\begin{equation}\label{eq:f_I_J_def}
  f_{I, J}(\pi) = 1_{\pi(I)=J}.
\end{equation}
where $I,J \in [n]_t$ (we changed notation for the subspace from $V$
to $W$ in this section to avoid confusion with notations arising
later which are related to the representation theory of the
symmetric group). With this choice, a subset $T\subset B$ satisfying
\eqref{eq:prob_def} is precisely a $t$-wise permutation. We now
establish the conditions of Theorem~\ref{thm:main}. We will show
that:
\begin{enumerate}
\item The divisibility constant of $W$ is $c_1=\frac{n!}{(n-t)!}$.
\item $W$ has a $c_2$-bounded integer basis in $\ell_{\infty}$ with $c_2=1$.
\item $W^{\perp}$ has a $c_3$-bounded integer basis in $\ell_{1}$ with $c_3=(t+2)!$.
\item The group $S_n$ acts on $W$ transitively.
\item The space $W$ contains the constant functions.
\item The dimension of $W$ equals the number of permutations in
$S_n$ with longest increasing subsequence of length at least $n-t$.
It satisfies $\dim(W)\le |[n]_t|^2\le n^{2t}$.
\end{enumerate}
With these conditions, Theorem~\ref{thm:main} immediately implies
Theorem~\ref{thm:perms_first_version}.

A few conditions are easy to verify. First, $W$ contains the
constant functions since for each $I\in[n]_t$, $\sum_{J \in [n]_t}
f_{I,J}$ is the constant function $1$. Second, any spanning subset
of the functions $f_{I,J}$ is an integer basis for $W$ with
$\ell_{\infty}$ norm $c_2=1$. Third, observe that $S_n$ acts
naturally on $B=S_n$ by composition. Each $\sigma\in S_n$ is a
symmetry of $W$ since for any $I,J\in [n]_t$,
\begin{equation*}
  (f_{I,J}\circ\sigma)(\pi)=f_{I,J}(\sigma(\pi)) =
  1_{\sigma(\pi(I))=J} =
  1_{\pi(I)=\sigma^{-1}(J)}=f_{I,\sigma^{-1}(J)}(\pi)\in W.
\end{equation*}
The action of $S_n$ on $B$ is transitive, from which the symmetry
condition follows. Fourth, $\dim(W)\le |[n]_t|^2$ since the
$(f_{I,J})$ span $W$.

In order to find the divisibility constant of $W$ and establish that
$W^{\perp}$ is spanned by integer vectors with small $\ell_1$ norm,
we will need some basic facts regarding the irreducible
representations of the symmetric group, which we describe next. We
will follow the notation of \cite[Chapter 4]{FultonHarris}, but also
refer the reader to \cite{James} for the necessary background.

\subsubsection{Irreducible representations of the symmetric group}\label{sec:defs}

\paragraph{Partitions.}
A \emph{partition} $\lambda$ of $n$ is a vector $(\lambda_1, \ldots,
\lambda_\ell)$ of some length $\ell$ such that
$\lambda_1\ge\cdots\ge\lambda_\ell\ge 1$ and $\sum_{i=1}^\ell
\lambda_i = n$. If referring to $\lambda_i$ for $i>\ell$ we adapt
the convention that $\lambda_i=0$ for such $i$. The \emph{conjugate
partition} $\lambda'$ is defined as $(\lambda'_1, \ldots,
\lambda'_m)$ where $m=\lambda_1$ and $\lambda'_i:=|\{j\ :\
\lambda_j\ge i\}|$. The \emph{dominance partial order} $\unrhd$ on
partitions is defined by
\begin{equation*}
  \lambda \unrhd \mu\quad\text{if and only if}\quad \sum_{i=1}^j \lambda_i\ge
  \sum_{i=1}^j \mu_i\quad\text{for all $j\ge 1$.}
\end{equation*}
We let $\ge$ stand for the \emph{lexicographic total order} on
partitions. It is well-known that the lexicographic order extends
the dominance order, and that $\lambda\unrhd\mu$ if and only if
$\mu'\unrhd\lambda'$. We denote by $\PP_{n}$ the set of all partitions of $n$.

\paragraph{Young diagrams, tableaux and tabloids.}
Let $\lambda \in \PP_n$. Associated with it is the \emph{Young
diagram} of shape $\lambda$ (in English notation).  A \emph{tableau}
of shape $\lambda$ is a filling of the Young diagram of shape
$\lambda$ with the integers $1$ to $n$, with each integer occuring
once. We say that two tableaux are \emph{row-equivalent} if they are
the same up to the order of the numbers in each row. A
\emph{tabloid} of shape $\lambda$ is an equivalence class of
tableaux for the row-equivalence relation. We denote the tabloid
associated to the tableau $T$ by $[T]$. $S_n$ acts on tableaux with
the permutation action on the numbers in each tableau. This induces
a corresponding action on tabloids. The \emph{column stabilizer} of
a tableau $T$ of shape $\lambda$ is the subgroup $Q_T\le S_n$ of
permutations preserving the columns of $T$.

\paragraph{Irreducible representations, Young modules and Kostka
numbers.} The irreducible representations (over $\C$ or $\Q$) of the
symmetric group $S_n$ are in correspondence with shapes
$\lambda\in\PP_n$. We denote by $V_\lambda$ the irreducible
representation corresponding to $\lambda$. We let $U_\mu$, sometimes
called the Young module or permutation module of shape $\mu$, be the
module whose basis is all tabloids of shape $\mu$, equipped with the
action of $S_n$ on tabloids. Each Young module $U_{\mu}$ is
isomorphic to a sum of irreducible representations. The \emph{Kostka
number} $K_{\lambda,\mu}$ is the multiplicity of the irreducible
representation $V_\lambda$ in $U_\mu$. It is known that
$K_{\lambda,\mu}=0$ unless $\lambda\unrhd \mu$ and
$K_{\lambda,\lambda}=1$.

\paragraph{The group algebra and Fourier analysis.}
We denote by $\C S_n$ the group algebra of $S_n$, the set of
functions $f: S_n\to\C$ endowed with the product
$$
(f * g)(\pi) = \sum_{\sigma \in S_n} f(\sigma) g(\sigma^{-1} \pi).
$$
We fix once and for all a matrix representation for each irreducible
representation $V_\lambda$. Then the functions
$V_\lambda(\cdot)_{i,j}:S_n\to\C$ for $\lambda\in\PP_n$ and $1\le
i,j\le \dim(V_\lambda)$ are linearly independent and span $\C S_n$.
Moreover, extending $V_\lambda(\cdot)$ linearly to all of $\C S_n$,
$$
V_\lambda(f * g)= V_\lambda(f) V_\lambda(g)\quad\text{ for $f,g\in\C
S_n$ and $\lambda\in\PP_n$}.
$$
In addition, $f=0$ if and only if $V_\lambda(f)=0$ for all
$\lambda\in\PP_n$.

\subsubsection{Two bases for $W$ and the divisibility
constant}\label{sec:bases_for_W}

Define $W'$ to be the span of the functions
$\{V_{\lambda}(\cdot)_{i,j}: \lambda_1 \ge n-t, 1 \le i,j \le
  \dim(V_{\lambda})\}$. As all the $(V_\lambda(\cdot)_{i,j})$ are linearly independent, we have
\begin{equation}\label{eq:dim_W'}
  \dim(W')=\sum_{\lambda \in \PP_n: \lambda_1 \ge n-t}\dim(V_{\lambda})^2.
\end{equation}
  In this section we show that $W=W'$ and
  define a combinatorial basis which is useful in determining the
  divisibility constant of $W$.
\begin{claim}\label{cl:W_is_contained_in_W'}
$W\subseteq W'$.
\end{claim}
\begin{proof}
  It suffices to show that $f_{I,J}\in W'$ for all $I,J\in[n]_t$.
  Consider the representation $R_t$ of the action of $S_n$ on
  $t$-tuples, defined by
  \begin{equation*}
    R_t(\pi)_{I,J} = 1_{\pi(I)=J} = f_{I,J}(\pi) \quad \text{ for }\pi\in S_n\text{ and } I,J\in[n]_t.
  \end{equation*}
  It is simple to see that $R_t$ is isomorphic to the Young
  module $U_{\mu}$ for $\mu=(n-t,1,1,\ldots,1) \in \PP_n$. Indeed,
  this follows by identifying each $I \in [n]_t$ with the tabloid of
  shape $\mu$ having the elements of $I$, in order, as the elements
  in its first column at rows $2,\ldots, t+1$. By definition of the
  Kostka number $K_{\lambda,\mu}$, if the irreducible representation
  $V_\lambda$ appears in the decomposition of $U_\mu$ then $\lambda\unrhd
  \mu$, which occurs if and only if $\lambda_1\ge n-t$. In
  conclusion, the decomposition of $R_t$ into irreducible
  representations contains only $V_\lambda$ with $\lambda_1\ge n-t$.
  Passing to the basis of irreducible representations, we conclude
  that $f_{I,J}\in W'$ for each $I,J\in [n]_t$, as required.
\end{proof}
It is not difficult to use the same representation-theoretic methods
to show that $W=W'$, but we proceed by a different route in order to
identify also the divisibility constant of $W$.

For $\sigma\in S_n$, let $\LIS(\sigma)$ denote the length of the
longest increasing subsequence of $\sigma$. Define
\begin{equation}\label{eq:combinatorial_basis_A}
  A := \{\sigma\in S_n\,:\, \LIS(\sigma)\ge n-t\}.
\end{equation}
\begin{claim}\label{claim:dimW_A}
  $|A|=\dim(W')$.
\end{claim}
\begin{proof}
The Robinson–-Schensted correspondence~\cite{Robinson38,Schensted61}
shows that
$$
\left|\{\sigma\in S_n\,:\, \LIS(\sigma) = r\}\right| = \sum_{\lambda
\in \PP_n: \lambda_1 = r} \dim(V_{\lambda})^2.
$$
Summing over $n-t \le r \le n$ and comparing with \eqref{eq:dim_W'}
concludes the proof.
\end{proof}
We now define a set of functions $(f_\sigma)$, $\sigma\in A$,
forming a basis of $W$. For each $\sigma\in A$, let
$S(\sigma)\subseteq [n]$ be the indices of an (arbitrary) increasing
subsequence in $\sigma$ of length $\LIS(\sigma)$. That is,
$S(\sigma)=(i_1,\ldots, i_{\LIS(\sigma)})$ for some indices
satisfying $i_{j+1}>i_j$ and $\sigma(i_{j+1})>\sigma(i_j)$. Define
\begin{equation}\label{eq:f_sigma_def}
  f_\sigma(\pi) = \begin{cases} 1& \pi(j) = \sigma(j)\quad \forall j\notin S(\sigma)\\ 0&\text{otherwise}\end{cases}.
\end{equation}
It is clear that the functions $f_\sigma$ are in $W$ since
$|S(\sigma)|\ge n-t$. Let $\succeq$ stand for the lexicographic
order on $S_n$. I.e., $\pi\succ\sigma$ if there exists a $j$ such
that $\pi(i)=\sigma(i)$ for all $i<j$ and $\pi(j)>\sigma(j)$.
\begin{lemma}\label{lemma:twise-basis-order}
  For each $\sigma\in A$, $f_\sigma(\sigma)=1$ and if
  $f_\sigma(\pi)=1$ then $\pi\succeq\sigma$.
\end{lemma}
\begin{proof}
  $f_\sigma(\sigma)=1$ by the definition of $f_\sigma$. Suppose that
  $f_\sigma(\pi)=1$. Then $\pi(i)=\sigma(i)$ for every $i\notin
  S(\sigma)$ and $\{\pi(i)\,:\,i\in S(\sigma)\}=\{\sigma(i)\,:\,i\in S(\sigma)\}$.
  Since $S(\sigma)$ are the indices of an increasing subsequence
  in $\sigma$ it follows that $\pi$ appears after $\sigma$ in the
  lexicographic order.
\end{proof}
Now let $\phi:B\to\Z^A$ be the matrix whose columns are the
$f_\sigma$. It follows from Lemma~\ref{lemma:twise-basis-order} that
when the rows and columns are sorted by the lexicographic order on
permutations, then $\phi$ is in column-echelon form. Consequently,
the columns of $\phi$ are linearly independent and hence $\dim(W)\ge
|A|$. Combining this fact with Claims~\ref{cl:W_is_contained_in_W'}
and \ref{claim:dimW_A} shows that $W=W'$ and that the columns of
$\phi$ form a basis for $W$. In addition, the column-echelon form of
$\phi$ implies that $\LL(\phi) = \Z^A$. Now, if $\sigma\in A$ has
$\LIS(\sigma)=n-\ell$ for some $0\le\ell\le t$ then
\begin{equation*}
  \frac{1}{|B|}\sum_{\pi\in B}\phi(\pi)_\sigma = \frac{1}{n!}\left|\{\pi \in S_n:
\pi(j)=\sigma(j) \quad \forall j \notin S(\sigma)\}\right| =
\frac{(n-\ell)!}{n!}.
\end{equation*}
Hence $\frac{N}{|B|}\sum_{b\in B} \phi(b)\in\Z^A=\LL(\phi)$ if and
only if $N$ is a multiple of $\frac{n!}{(n-t)!}$, implying that
$c_1=\frac{n!}{(n-t)!}$ is the divisibility constant of $W$.

We end the section by giving an alternate characterization of $W$.
\begin{claim}\label{cl:W_characterization}
  A function $f\in\C S_n$ satisfies $f\in W$ if and only if
  $V_\lambda(f)=0$ for every $\lambda\in \PP_n$ with $\lambda_1\le
  n-t-1$.
\end{claim}
\begin{proof}
We first recall the orthogonality relations for irreducible
representations which state that
\begin{equation}\label{eq:orthogonality_relations}
  \sum_{\pi\in S_n} V_\mu(\pi^{-1})_{i,j} V_\lambda(\pi)_{k,\ell} =
  \frac{n!}{\dim(V_\lambda)}\delta_{\lambda, \mu}\delta_{i,k}\delta_{j,\ell} \quad\text{ $\lambda,\mu\in\PP_n$, $1\le i,j\le \dim(V_\lambda)$, $1\le
  k,\ell\le \dim(V_\mu)$}.
\end{equation}
  We continue by observing that $f\in W$ if and only if $\tilde{f}\in
  W$ where $\tilde{f}$ is defined by $\tilde{f}(\pi)=f(\pi^{-1})$. Indeed, $f\in W$ if and only if $f = \sum_{I,J\in [n]_t}
  \alpha_{I,J} f_{I,J}$ for the functions $f_{I,J}$ defined by \eqref{eq:f_I_J_def} and some coefficients
  $\alpha_{I,J}$. The observation now follows since
  $f_{I,J}(\pi^{-1}) = f_{J,I}(\pi)$.

  Now let $f\in\C S_n$. Decompose $\tilde{f}$ in the basis of irreducible
representations of $S_n$ as
\begin{equation*}
\tilde{f}(\pi) = \sum_{\mu \in \PP_n} \sum_{i,j=1}^{\dim(V_{\mu})}
\alpha_{\mu,i,j} \cdot V_{\mu}(\pi)_{i,j},
\end{equation*}
where $\alpha_{\mu,i,j} \in \C$. Then,
\begin{equation*}
  V_\lambda(f) = \sum_{\pi\in S_n}\sum_{\mu \in \PP_n} \sum_{i,j=1}^{\dim(V_{\mu})}
\alpha_{\mu,i,j} \cdot V_{\mu}(\pi^{-1})_{i,j}V_\lambda(\pi).
\end{equation*}
Thus, \eqref{eq:orthogonality_relations} implies that
\begin{equation*}
V_\lambda(f)_{k,\ell}=
\frac{\alpha_{\lambda,k,\ell}n!}{\dim(V_\lambda)}.
\end{equation*}
We conclude that $V_\lambda(f)=0$ for all $\lambda$ with
$\lambda_1\le n-t-1$ if and only if $\alpha_{\mu,i,j}=0$ for all
$\mu$ with $\mu_1\le n-t-1$. That is, if and only if $\tilde{f}\in
W'$. Since $W=W'$ and $\tilde{f}\in W$ if and only if $f\in W$ we
conclude that $V_\lambda(\tilde{f})=0$ for all $\lambda$ with
$\lambda_1\le n-t-1$ if and only if $f\in W$.
\end{proof}

\subsubsection{Antisymmetrizers}
We introduce the \emph{column antisymmetrizer} of a tableau $T$,
\begin{equation}\label{eq:b_t_def}
  b_T := \sum_{\sigma\in Q_T} \sign(\sigma)\cdot \sigma,
\end{equation}
which is an element of the group algebra $\C S_n$. We study in which
irreducible representations the antisymmetrizers have a non-trivial
action. We start by studying this question for the Young modules
since these have a simpler combinatorial nature.

\begin{claim}\label{claim:action_U_bT}
Let $T$ be a tableau of shape $\lambda \in \PP_n$. Then
\begin{enumerate}
\item[(i)] $U_{\lambda}(b_{T}) \ne 0$.
\item[(ii)] If $U_{\mu}(b_T) \ne 0$ then $\lambda \unrhd \mu$.
\end{enumerate}
\end{claim}

\begin{proof}
We first recall the basic properties. For tabloids $[T'],[T'']$ of shape $\mu$ and $\pi \in S_n$ we have
$$
U_{\mu}(\pi)_{[T'],[T'']} = 1_{\pi([T'])=[T'']}.
$$
We first establish $(i)$. We have
$$
U_{\lambda}(b_T)_{[T],[T]} = \sum_{\pi \in Q_T} \sign(\pi) 1_{\pi([T])=[T]} = 1,
$$
since any permutation in $Q_T$ except the identity maps $[T]$ to a
different tabloid. In particular $U_{\lambda}(b_T) \ne 0$.

We next establish $(ii)$. To do so, we show that for every tabloid
$[S]$ of shape $\mu$ we have $b_T [S]=0$. Assume for a moment that
there exist two elements $a,b$ in the same row of $S$ and the same
column of $T$. Define $s_{a,b} = \frac{1}{2}(\rm{Id} - (a,b)) \in \C
S_n$ where $\rm{Id}$ is the identity permutation and $(a,b)$ is the
permutation that swaps $a$ and $b$. On the one hand $b_T * s_{a,b} =
b_T$, and on the other hand $s_{a,b} [S]=0$. Hence
$$
b_T [S] = (b_T * s_{a,b}) [S] = b_T (s_{a,b} [S]) = 0.
$$
So, if $U_{\mu}(b_T) \ne 0$ such a pair cannot exist. Thus, all the elements in the first column of $T$ appear in different rows in $S$; all the elements in the second column of $T$ appear in different rows of $S$, etc. This implies that $\mu' \unrhd \lambda'$ where $\mu',\lambda'$ are the conjugate partitions to $\mu,\lambda$. This in turn implies that $\lambda \unrhd \mu$ as we claimed.
\end{proof}

We now derive the analogous claim for the irreducible
representations.

\begin{claim}\label{claim:action_V_bT}
Let $T$ be a tableau of shape $\lambda \in \PP_n$. Then
\begin{enumerate}
\item[(i)] $V_{\lambda}(b_{T}) \ne 0$.
\item[(ii)] If $V_{\mu}(b_T) \ne 0$ then $\lambda \unrhd \mu$.
\end{enumerate}
\end{claim}

\begin{proof}
Let $\mu\in\PP_n$ be such that $\lambda\!\! \not\!\unrhd\, \mu$.
$K_{\mu, \mu}=1$ and hence $V_\mu$ appears in the decomposition of
$U_\mu$ into irreducible representations. Thus, the fact that
$U_\mu(b_T)=0$ by Claim~\ref{claim:action_U_bT} implies that
$V_\mu(b_T)=0$, proving $(ii)$.

Now, since $K_{\lambda, \lambda}=1$ and $K_{\tau,\lambda}=0$ unless
$\tau\unrhd \lambda$ we have that $U_\lambda$ decomposes as the sum
of $V_\lambda$ plus other irreducible representations $V_\tau$ with
$\tau \rhd \lambda$. Since $U_\lambda(b_T)\neq 0$ by
Claim~\ref{claim:action_U_bT} and $V_\tau(b_T)=0$ when $\tau \rhd
\lambda$ by part $(ii)$ we deduce that $V_\lambda(b_T)\neq 0$,
proving $(i)$.
\end{proof}

\subsubsection{Spanning vectors for $W^{\perp}$}
We will prove the following lemma in this subsection.
\begin{lemma}\label{lemma:span_W_perp}
Let $f \in \C S_n$. Then $f \in W$ iff $b_T * f=0$ for all tableaux $T$ of shape $\lambda$ with $\lambda_1=n-t-1$.
\end{lemma}

We first show that this gives a basis of integer vectors for
$W^{\perp}$ of small $\ell_1$ norm.

\begin{cor}\label{cor:span_W_perp_bound}
$W^{\perp}$ has a $(t+2)!$-bounded integer basis in $\ell_1$.
\end{cor}

\begin{proof}
Let $T$ be a tableau of shape $\lambda \in \PP_n$ with $\lambda_1=n-t-1$. The condition $b_T * f=0$ is equivalent to
$$
\sum_{\sigma \in Q_T} \sign(\sigma) f(\sigma^{-1} \pi)=0\quad\text{
for all $\pi \in S_n$}.
$$
The $\ell_1$ norm of the vectors in $W^{\perp}$ these define is
$|Q_T|$, which we next derive a bound on. Let
$\lambda'=(\lambda'_1,\ldots,\lambda'_m)$ be the conjugate partition
to $\lambda$. Then $|Q_T| = \prod_{i=1}^m \lambda'_i!$. Observe that
for any $a,b\ge 0$, $b+1=\binom{b+1}{b}\le \binom{a+b+1}{b}$ and
hence $(a+1)!(b+1)!\le (a+b+1)!$. Thus
\begin{equation*}
  |Q_T| = \prod_{i=1}^m \lambda'_i! = \prod_{i=1}^{m} (\lambda'_i-1+1)! \le \left(1+\sum_{i=1}^{m} (\lambda'_i-1)\right)! =
  (1+n-\lambda_1)!=(t+2)!.\qedhere
\end{equation*}
\end{proof}
For the proof of Lemma~\ref{lemma:span_W_perp} we need the following
auxiliary claims.
\begin{claim}\label{cl:minimal_derivative}
  Let $\mu\in \PP_n$ satisfy $\mu_1\le n-t-1$. Then for every tableau $S$
  of shape $\mu$ there exists a shape $\lambda\in \PP_n$ with
  $\lambda_1=n-t-1$, a tableau $T$ of shape $\lambda$ and an element
  $g\in \C S_n$ such that $b_S = g*b_T$.
\end{claim}
\begin{proof}
Define an element in a tableau as {\em maximal} if it is the last
element in its row and its column (in English notation, it is the
rightmost element in its row and the bottom element in its column).
We construct $T$ from $S$ by iteratively moving maximal elements
which are not in the first row to the end of the first row, until
the first row of $T$ contains exactly $n-t-1$ elements. It is simple
to verify that this process guarantees that $Q_T$ is a subgroup of
$Q_S$, since each column of $T$ is contained in a column of $S$. Let
$\{\sigma_1,\ldots,\sigma_r\}$ be representatives for the left
cosets of $Q_T$ in $Q_S$. Then we have
\begin{equation*}
b_S = \left(\sum_{i=1}^r \sign(\sigma_i) \cdot \sigma_i\right) *
b_T.\qedhere
\end{equation*}
\end{proof}

Observe that if $T,S$ are tableaux of shape $\lambda$ and $\sigma\in
S_n$ satisfies $\sigma(T) = S$, then $b_S = \sigma b_T\sigma^{-1}$.
Define, for $\lambda\in\PP_n$,
\begin{equation}\label{eq:c_lambda_def}
  c_\lambda := \sum_{S\text{ of shape $\lambda$}} b_S =
  \sum_{\sigma\in S_n} \sigma b_T\sigma^{-1}
\end{equation}
where $T$ is an arbitrary tableau of shape $\lambda$. Denoting by
$\tr$ the trace of a matrix, the definition implies that
\begin{equation}\label{eq:c_lambda_trace}
  \tr(V_\mu(c_\lambda))=n!\tr(V_\mu(b_T))\quad\text{ for
  $\mu\in\PP_n$}.
\end{equation}
\begin{claim}\label{cl:multiple_of_identity}
  $V_\mu(c_\lambda)$ is a multiple of the identity for all
  $\lambda, \mu\in\PP_n$. Moreover, if there exists a tableau $T$ of
  shape $\lambda$ such that $V_\mu(b_T)\neq 0$ then
  $V_\mu(c_\lambda)\neq 0$.
\end{claim}
\begin{proof}
  Fix $\lambda, \mu\in\PP_n$. By \eqref{eq:c_lambda_def} we have for any $\pi\in S_n$,
  \begin{equation*}
    \pi c_\lambda = \sum_{\sigma\in S_n} \pi\sigma b_T\sigma^{-1} =
    \sum_{\sigma\in S_n} \sigma b_t \sigma^{-1} \pi = c_\lambda \pi.
  \end{equation*}
  Hence Schur's lemma implies that $V_\mu(c_\lambda)$ is a multiple
  of the identity.

  Now suppose that $V_\mu(b_T)\neq 0$ for some tableau $T$ of shape
  $\lambda$. Observe that $b_T*b_T = |Q_T| b_T$. Thus $|Q_T|^{-1} V_\mu(b_T)$ is a projection matrix. Hence, $V_\mu(b_T)\neq 0$ implies that $\tr(V_\mu(b_T))\neq 0$. Thus
  $\tr(V_\mu(c_\lambda))\neq 0$ by \eqref{eq:c_lambda_trace}, from
  which $V_\mu(c_\lambda)\neq 0$ follows.
\end{proof}

\begin{proof}[Proof of Lemma~\ref{lemma:span_W_perp}]
Suppose first that $f\in W$ and let $T$ be a tableau of shape
$\lambda$ for some $\lambda\in\PP_n$ satisfying $\lambda_1=n-t-1$.
We will show that $b_T*f=0$ by showing that $V_\mu(b_T*f)=0$ for all
$\mu\in\PP_n$. Fix $\mu\in\PP_n$. If $\mu_1\le n-t-1$ we have
$V_\mu(f)=0$ by Claim~\ref{cl:W_characterization}. If $\mu_1\ge n-t$
we have $V_\mu(b_T)=0$ by Claim~\ref{claim:action_V_bT}$(ii)$. Thus
in all cases $V_\mu(b_T*f) = V_\mu(b_T)V_\mu(f)=0$.

Now suppose that $f\notin W$. By Claim~\ref{cl:W_characterization}
there exists some $\mu\in\PP_n$ with $\mu_1\le n-t-1$ such that
$V_\mu(f)\neq 0$. Putting together
Claim~\ref{claim:action_V_bT}$(i)$ and
Claim~\ref{cl:multiple_of_identity} we have that $V_\mu(c_\mu)$ is a
non-zero multiple of the identity. Thus $V_\mu(c_\mu *
f)=V_\mu(c_\mu)V_\mu(f)\neq 0$. By the definition
\eqref{eq:c_lambda_def} of $c_\mu$, this implies that there exists
some tableau $S$ of shape $\mu$ such that $V_\mu(b_S * f)\neq 0$. By
Claim~\ref{cl:minimal_derivative}, there exists a shape $\lambda\in
\PP_n$ with $\lambda_1=n-t-1$, a tableau $T$ of shape $\lambda$ and
an element $g\in \C S_n$ such that $b_S = g*b_T$. Since
$V_\mu(g)V_\mu(b_T * f) = V_\mu(b_S
* f)\neq 0$, we conclude that $b_T
* f\neq 0$, as required.
\end{proof}

\subsection{The number of $t$-wise
permutations}\label{sec:number_t_wise_permutations} One may use
Theorem~\ref{thm:count_solutions} to estimate the number of $t$-wise
permutations of a given size, as we did for orthogonal arrays and
$t$-designs. To this end, one needs to calculate the parameter
$\rho(W)$ defined by \eqref{eq:rho_formula}. We leave this
calculation for future work but present in this section the results
of numerical calculations which give some evidence that $\rho(W)$
has a nice product structure.

Let $B=S_n$ and $A$ be as in \eqref{eq:combinatorial_basis_A}. Let
$\phi$ be the $B\times A$ matrix whose columns are given by the
$(f_\sigma)$, $\sigma\in A$, defined in \eqref{eq:f_sigma_def}. As
proven in Section~\ref{sec:bases_for_W}, the columns of $\phi$ form
a basis for $W$ and $\LL(\phi)=\Z^A$. Thus
\begin{equation*}
  \rho(W) = \frac{\det(\LL(\phi))}{\sqrt{\det(\phi^t\phi)}} =
  \frac{1}{\sqrt{\det(\phi^t\phi)}}
\end{equation*}
for this matrix $\phi$. Below we present the results of numerical
calculations of $\det(\phi^t\phi)$ for a few small values of $n$ and
$t$.
\begin{equation*}
\begin{array}{lllll}
  n\backslash t & 1 & 2 & 3\\
  3     & 3\cdot2 & 1 & - \\
  4     & 3\cdot2^{18} & 3\cdot2^3 & 1 \\
  5     & 5^{9} 3^{17} 2^{19} & 5^9 3^2 2^{84} & 5\cdot3\cdot2^3 \\
  6     & 5\cdot3^{42}2^{94} & 5^{64}3^{162}2^{276} & ? \\
  7     & 7^{25}5^{37}3^{38}2^{112} & 7^{100}5^{65}3^{627}2^{1150} & ? \\
  8     & 7\cdot5^{50}3^{100}2^{308} & ? & ? \\
\end{array}
\end{equation*}

\section{Proof of main theorems}
\label{sec:proof_main_theorem}

We prove our main theorems, Theorem~\ref{thm:main} and
Theorem~\ref{thm:count_solutions}, in this section. We start by
stating a local central limit theorem from which our main theorems
will follow.

\subsection{Local central limit theorem
statement}\label{sec:LCLT_with_basis} Let $B$ be a finite set and
$V$ be a vector space of functions from $B$ to the rational numbers
$\Q$. Let $\{\phi_a:B \to \Q\}_{a \in A}$ be a basis for $V$, where
$A$ is some finite index set of size $\dim(V)$. This basis is
arbitrary for now but will be chosen in a convenient way in the next
subsection. Let $\phi:B \to \Q^A$ be defined as $\phi(b)_a =
\phi_a(b)$. It may be useful to think of $\phi$ as a $B\times A$
matrix, whose entries are $\phi_a(b)$. Fix $0<p<1$ and define $T$ to
be a random subset of $B$, with each point of $B$ chosen
independently into $T$ with probability $p$. In other words, we let
$\{T_b\}$, $b\in B$, be a collection of independent identically
distributed random variables with $\Pr[T_b=1]=1-\Pr[T_b=0]=p$ and
let $T:=\{b\in B\,:\, T_b=1\}$. Define
\begin{equation}\label{eq:X_def}
X:=\sum_{b \in B} T_b \cdot \phi(b) \in \LL(\phi),
\end{equation}
where $\LL(\phi)$ is the lattice in $\Q^A$ generated by
$\{\phi(b)\}$, $b\in B$. Our main Theorems will follow from a
precise estimate of the probability $\Pr[X = \Ex[X]]$. Along the
way, however, we will pass through estimating $\Pr[X=\lambda]$ for
an arbitrary point $\lambda\in\LL(\phi)$ (though our estimate will
only be meaningful for $\lambda$ close to $\Ex[X]$). Since this is a
useful result in itself, which also requires less assumptions, we
encapsulate it in the following theorem. We note that the mean of
$X$ is given by
\begin{equation*}
  \Ex[X] = p\sum_{b\in B} \phi(b)
\end{equation*}
and the covariance matrix of $X$ is given by
\begin{equation}\label{eq:Sigma_X_def}
  \Sigma[X] := \Ex[(X-\Ex[X])^t(X-\Ex[X])] = p(1-p) \phi^t\phi
\end{equation}
where $\phi^t\phi$ is the symmetric positive definite $A \times A$
matrix satisfying $(\phi^t\phi)_{a,a'} = \sum_{b \in B} \phi(b)_a
\phi(b)_{a'}$. The positive definite property follows from the fact
that the $\{\phi_a\}$, $a\in A$, are linearly independent.

\begin{thm}[Local central limit theorem]\label{thm:LCLT_with_basis}
There exists a constant $C>0$ such that the following is true.
Assume that the following conditions hold for some integers $c_2,c_3
\ge 1$,
\begin{enumerate}
\item Boundedness of $V$: $V$ has a $c_2$-bounded integer basis in $\ell_{\infty}$.
\item Boundedness of $V^{\perp}$: $V^{\perp}$ has a $c_3$-bounded integer basis in $\ell_{1}$.
\item Symmetry: for any $b_1,b_2 \in B$ there exists a symmetry $\pi$ of $V$ satisfying $\pi(b_1)=b_2$.
\end{enumerate}
If
\begin{equation*}
  \min(p|B|, (1-p)|B|) \ge C \cdot c_2 c_3^2\dim(V)^6\log(2c_3\dim(V))^6
\end{equation*}
then for every $\lambda\in\LL(\phi)$,
\begin{equation}\label{eq:LCLT_prob_estimate}
  \Pr[X=\lambda] =
  \frac{\det(\LL(\phi))}{(2\pi)^{\frac{\dim(V)}{2}}\sqrt{\det{\Sigma[X]}}}\left(e^{-\frac{1}{2}(\lambda-\Ex[X])^t\Sigma[X]^{-1}(\lambda-\Ex[X])}+\delta(\lambda)\right)
\end{equation}
with $|\delta(\lambda)|\le \frac{C\dim(V)^3
(\log(2c_2\dim(V)))^{3/2}}{\sqrt{\min(p|B|,(1-p)|B|)}}$.
\end{thm}
We point out explicitly that this theorem does not require the
divisibility or constant functions assumptions of our main theorems,
Theorem~\ref{thm:main} and Theorem~\ref{thm:count_solutions}. In
Subsection~\ref{sec:deduction_main_thms} below we explain how our
main theorems follow (easily) from this local central limit theorem
and the extra assumptions by applying the theorem with
$\lambda=\Ex[X]$.

We also note that the local central limit theorem does not depend on
our choice of basis $\phi$ in the sense that if it holds for one
basis it holds for all bases. We make this fact more explicit in
Subsection~\ref{sec:LCLT_basis_free} where we state an equivalent
basis-free version of the theorem.

The local central limit theorem is proved in
Subsections~\ref{sec:Fourier_analysis} to
\ref{sec:proof_of_main_thm_from_lemmas} below.

\subsection{Fourier analysis}\label{sec:Fourier_analysis}

We continue with the notation of the previous section and assume the
conditions of Theorem~\ref{thm:LCLT_with_basis}. We fix
$\{\phi_a\}$, $a\in A$ to be the basis of integer-valued functions,
satisfying $\|\phi_a\|_{\infty}\le c_2$ for all $a\in A$, whose
existence is guaranteed by the boundedness condition for $V$. We
also make the simplifying assumption
\begin{equation*}
  p\le \frac{1}{2}.
\end{equation*}
Near the end of the proof we will show how to get rid of this
assumption by utilizing the bijection $T\mapsto B\setminus T$.
Finally, we denote
\begin{equation*}
  N:=p|B|.
\end{equation*}
We stress that $N$ need not be an integer in our proof of the local
central limit theorem. However, when we later deduce our main
theorems from the local central limit theorem, we will choose $p$ in
such a way that $N$ will be an integer.

Our main technique to study the distribution of $X$ is Fourier
analysis. The Fourier transform of $X$ is the function
$\hat{X}:\R^A\to\C$ defined by
\begin{equation*}
  \hat{X}(\theta) := \Ex[e^{2 \pi i \ip{X,\theta}}],
\end{equation*}
where $\ip{X,\theta}:=\sum_{a \in A} X_a \theta_a$. Define the dual
lattice $L$ to $\LL(\phi)$ (sometimes called the annihilator or
reciprocal lattice of $\LL(\phi)$), as the set of vectors in $\R^A$
having an integer inner product with the vectors of $\LL(\phi)$.
That is,
$$
L := \{\theta \in \R^A\, :\, \ip{\theta,\lambda} \in \Z \quad
\forall \lambda \in \LL(\phi)\}.
$$
Noting that $\LL(\phi)$ has full rank, since the $\{\phi_a\}$, $a\in
A$, are linearly independent, it follows that $L$ is also a full
rank lattice and the relation $\det(\LL(\phi))\det(L)=1$ holds.
Since $e^{2\pi i \ip{X,\alpha}}=1$ almost surely when $\alpha\in L$,
we see that $\hat{X}$ is $L$-periodic,
\begin{equation}\label{eq:FT_periodicity}
    \hat{X}(\theta+\alpha) = \hat{X}(\theta)\quad\forall\theta\in\R^A,\, \alpha\in L.
\end{equation}
The covariance matrix of $X$ provides a natural norm to work with in
Fourier space. Define
$$
R:= \phi^t\phi,
$$
so that $\Sigma[X]=p(1-p)R$ by \eqref{eq:Sigma_X_def}. As mentioned,
$R$ is a symmetric positive definite $A \times A$ matrix satisfying
$R_{a,a'} = \sum_{b \in B} \phi(b)_a \phi(b)_{a'}$. We define a norm
in Fourier space by
\begin{equation*}
  \|\theta\|_R := \left(\frac{1}{|B|}\theta^t R\theta\right)^{1/2} = \left(\frac{1}{|B|}\sum_{b \in B}
\ip{\phi(b),\theta}^2\right)^{1/2}\qquad(\theta\in\R^A).
\end{equation*}
Balls in the $R$-norm are denoted by
$$
\B_{R}(\eps) := \{\theta \in \R^A\,:\, \|\theta\|_R\le \eps\}.
$$
Let $D$ be the Voronoi cell of $0$ in the lattice $L$, with respect
to the $R$-norm. That is,
\begin{equation}\label{eq:Voronoi_cell}
  D:=\{\theta\in\R^A\,:\, \|\theta\|_R< \|\theta-\alpha\|_R\quad
\forall \alpha \in L\setminus\{0\}\}.
\end{equation}
Observe that $D$ is a bounded set since $L$ has full rank. Moreover,
$\alpha+D$ and $\alpha'+D$ are disjoint for distinct
$\alpha,\alpha'\in L$, and $\cup_{\alpha\in L} (\alpha+D)$ covers
all of $\R^A$ except a set of Lebesgue measure zero (since only a
Lebesgue measure zero of points in $\R^A$ are equidistant to two
points in $L$). It follows that
$\vol(D)=\det(L)=\det(\LL(\phi))^{-1}$, where $\vol$ denotes
Lebesgue measure, and that we have the following inversion formula.
\begin{fact}[Fourier inversion formula on
lattices]\label{fact:Fourier_inversion}
$$
\Pr[X=\lambda] = \det(\LL(\phi))\int_D \hat{X}(\theta)
    e^{-2 \pi i \ip{\lambda,\theta}} d
    \theta\quad\forall\lambda\in\LL(\phi).
$$
\end{fact}

Thus, our goal from now on is to understand the Fourier transform of
$X$. We start with an explicit formula for $\hat{X}$.

\begin{claim}\label{cl:structure_fourier_X}
We have
$$
\hat{X}(\theta) = \prod_{b \in B}
    \left(1-p+p e^{2 \pi i \cdot \ip{\phi(b),\theta}}\right).
$$
\end{claim}

\begin{proof}
By definition $X=\sum_{b \in B} T_b \phi(b)$, where $T_b \in
\{0,1\}$ are independent with $\Pr[T_b=1]=p$. Thus
\begin{align*}
\hat{X}(\theta) &= \Ex[e^{2 \pi i \ip{X,\theta}}] = \Ex
    [e^{2 \pi i \sum_{b \in B} T_b \ip{\phi(b),\theta}}]\\
&= \prod_{b \in B} \Ex[e^{2 \pi i \; T_b \ip{\phi(b),\theta}}] =
\prod_{b \in B} (1-p+pe^{2 \pi i \ip{\phi(b),\theta}}).\qedhere
\end{align*}
\end{proof}

An important ingredient in controlling $\hat{X}$ is the following
property. The terms $\ip{\phi(b),\theta}$ which arise in the Fourier
transform $\hat{X}(\theta)$ are tame in the following sense: if most
of them are small, then all of them are small; and if most of them
are close to integers, then all of them are close to integers. This
is captured by the following lemma. The constant $c_2$ appearing in
the lemma is the one given in the boundedness assumption for $V$.
\begin{lemma}\label{lemma:phi_FC_tame}
There exists a universal constant $C>0$ such that if we set
\begin{equation}\label{eq:M_def}
M:=C(|A| \log(2c_2|A|))^{3/2}
\end{equation}
then for every $\theta \in \R^A$:
\begin{enumerate}
\item
$$
\max_{b \in B} |\ip{\phi(b),\theta}| \le  M
\left(\frac{1}{|B|}\sum_{b \in B} \ip{\phi(b),\theta}^2\right)^{1/2}
= M\|\theta\|_R.
$$

\item Write $\ip{\phi(b),\theta}=n_b+r_b$, where $n_b \in \Z$ and $r_b \in [-1/2,1/2)$. Then
$$
\max_{b \in B} |r_b| \le  M \left(\frac{1}{|B|}\sum_{b \in B}
r_b^2\right)^{1/2}.
$$
\end{enumerate}
\end{lemma}
We note that the proof of the lemma uses only the boundedness
assumption for $V$ and the assumption that $V$ has a transitive
symmetry group and it is the only place where the symmetry
assumption is used. We prove Lemma~\ref{lemma:phi_FC_tame} in
Subsection~\ref{subsec:tame}. The main ingredient in its proof is
the notion of local correctability of the map $\phi$.

Our next step is to approximate $\hat{X}$ near zero. The next lemma
achieves this by approximating $\hat{X}(\theta)$ by its Taylor
expansion at zero for $\theta\in\B_R(\eps)$.

\begin{lemma}[Estimating the Fourier transform near zero]
\label{lemma:estimate_fourier_near_zero} For all
$0<\eps\le\frac{1}{8M}$ and $\theta\in \B_R(\eps)$,
$$
\hat{X}(\theta) = \exp(2 \pi i \cdot \ip{\Ex[X], \theta} - 2 \pi^2
\cdot \theta^t \Sigma[X] \theta  + \delta(\theta))
$$
where $|\delta(\theta)|= O(M \|\theta\|_R^3 N)$.
\end{lemma}
In this lemma as well as in the remainder of the paper we use the
$O(\cdot)$ notation to hide universal constants, independent of all
other parameters. We prove
Lemma~\ref{lemma:estimate_fourier_near_zero} in
Subsection~\ref{subsec:fourier_near_zero}. We next derive an upper
bound on the Fourier transform at points which are far from zero.
Recalling that $\hat{X}(\theta)$ is an $L$-periodic function, such a
bound can only hold for $\theta$ bounded away from the points of
$L$. We achieve this by requiring $\theta$ to belong to
$D\setminus\B_R(\eps)$.

\begin{lemma}[Bounding the Fourier transform far from $L$]
\label{lemma:estimate_fourier_far_L} For all $\eps>0$ and $\theta\in
D\setminus \B_R(\eps)$,
$$
|\hat{X}(\theta)| \le \exp(-\beta^2 N)
$$
where $\beta = \beta(\eps) = \min(\eps, \frac{1}{c_3 M})$.
\end{lemma}
We prove Lemma~\ref{lemma:estimate_fourier_far_L} in
Subsection~\ref{subsec:fourier_far_L}. This lemma is the only place
where we use the boundedness assumption for $V^\perp$.

\subsection{Local correction}
\label{subsec:tame}

A map $\psi:B \to \Z^A$ is said to be locally correctable, if for
any small subset $E \subset B$ and any $e \in E$, we can express
$\psi(e)$ as a short integer combination of $\{\psi(b): b \in B
\setminus  E\}$. This is an analog of the local correction property
of codes, usually studied over finite fields.

\begin{dfn}[Locally correctable]
A map $\psi:B \to \Z^A$ is called {\em $(\delta,s)$-locally
correctable} if for any $E \subset B$ of size $1\le|E| \le \delta
|B|$ and any $e \in E$, there exists a $\gamma \in \Z^{B \setminus
E}$ with $\|\gamma\|_1 \le s$ such that
$$
\psi(e) = \sum_{b \in B \setminus E} \gamma_b \cdot \psi(b).
$$
\end{dfn}
Regarding $\psi$ as a $B\times A$ matrix, we see that local
correctability is actually a property of the space $W$ spanned by
the columns of $\psi$ and does not depend on the particular choice
of basis given by $\psi$. Still, it is convenient to define local
correctability this way since our usage for it will be with a
particular choice of basis. We say that a permutation $\pi\in S_B$
is a \emph{symmetry of $\psi$} if it is a symmetry of $W$. We show
that the assumptions that $\psi$ is bounded and has a transitive
symmetry group imply that it is locally correctable. We use the
notation $\|\psi\|_{\infty}:=\max_{a\in A, b\in B}|\psi(b)_a|$,
$\psi(\gamma):=\sum_{b\in B}\gamma_b\psi(b)$ for $\gamma\in\R^B$ and
$\psi(S)=\sum_{b\in S} \psi(b)$ for $S\subset B$.

\begin{lemma}\label{lemma:phi_is_locally_correctable}
Let $\psi:B \to \Z^A$ be such that $\|\psi\|_{\infty} \le c$ and
such that the symmetry group of $\psi$ acts transitively on $B$.
Then $\psi$ is $(\delta, s)$-locally correctable for some $s = O(|A|
\log(2c|A|))$ and $\delta = \frac{1}{8s}$.
\end{lemma}

We need the following auxiliary claim.

\begin{claim}\label{claim:phi_has_short_orth_vecs}
Let $\psi:B \to \Z^A$ be such that $\|\psi\|_{\infty} \le c$. Then
for any subset $S \subset B$ of size $|S| \ge O(|A| \log(2c |A|))$
there exists a vector $\gamma \in \{-1,0,1\}^S$ having at least
$\frac{1}{4}|S|$ non-zero coordinates and satisfying
$\psi(\gamma)=0$.
\end{claim}

\begin{proof}
The claim follows from the pigeon hole principle. Let $\alpha\in
(0,1)$ be such that for any $n\ge1$, the number of strings in
$\{0,1\}^{n}$ having less than $\frac{1}{4}n$ ones is at most
$2^{\alpha n}$. Fix $S\subset B$. For a subset $S' \subseteq S$, we
have that $\psi(S')$ is an $|A|$-dimensional integer vector with
entries bounded by $c |S|$ in absolute value and hence the total
number of distinct values for it is bounded by $(2c|S|+1)^{|A|}$.
The number of subsets of $S$ is $2^{|S|}$. Hence, if
\begin{equation}\label{eq:volume_condition}
2^{(1-\alpha)|S|} > (2c|S|+1)^{|A|}
\end{equation}
there must exist two distinct subsets $S_1,S_2$ such that
$\psi(S_1)=\psi(S_2)$ and $|S_1\bigtriangleup S_2|\ge\frac{1}{4}|S|$
(where we use $\bigtriangleup$ to denote symmetric difference). We
then set $\gamma = 1_{S_1} - 1_{S_2}$ and have that $\psi(\gamma)=0$
and $\gamma$ has at least $\frac{1}{4}|S|$ non-zero coordinates.
Now, it is a simple exercise to verify that the condition $|S| \ge
O(|A| \log(2c |A|))$ with a large enough hidden constant implies
\eqref{eq:volume_condition}.
\end{proof}

\begin{proof}[Proof of Lemma~\ref{lemma:phi_is_locally_correctable}]
Let $s=O(|A| \log(2c |A|))$ be an integer larger than the lower
bound on $|S|$ given by Claim~\ref{claim:phi_has_short_orth_vecs}.
Set $\delta=\frac{1}{8 s}$. Assume that $\delta|B|\ge 1$ since
otherwise the claim holds vacuously. Fix a subset $E \subset B$ of
size $1\le|E| \le \delta |B|$ and an element $e \in E$. For each
$b\in B$, let $\pi_b$ be a symmetry of $\psi$ satisfying $\pi_b(b) =
e$.

Choose $S$ uniformly among the subsets of size $s$ of $B$ and choose
$b_0$ uniformly in $S$. Let $\gamma\in\{-1,0,1\}^S$ be some (random)
vector having at least $\frac{1}{4}|S|$ non-zero coordinates and
satisfying $\psi(\gamma)=0$, whose existence is guaranteed by
Claim~\ref{claim:phi_has_short_orth_vecs}. Define the events
\begin{align*}
  \Omega_1&:=\{\gamma_{b_0}=0\},\\
  \Omega_2&:=\{\exists b\in S\setminus\{b_0\}\text{ satisfying }\pi_{b_0}(b)\in E\}.
\end{align*}
We first show that $\Omega_1^c\cap\Omega_2^c$ has positive
probability. Indeed, this follows from the fact that
$\Pr[\Omega_1]\le \frac{3}{4}$ by the properties of $\gamma$ and,
observing that conditioned on $b_0$ each $b\in S\setminus\{b_0\}$ is
uniformly distributed in $B\setminus\{b_0\}$,
\begin{equation*}
  \Pr[\Omega_2] \le s\frac{|E|}{|B|}\le s\delta= \frac{1}{8}.
\end{equation*}
Now, assume that $\Omega_1^c\cap\Omega_2^c$ occurred. We may assume
WLOG that $\gamma_{b_0}=1$ (since otherwise we may replace $\gamma$
by $-\gamma$). Thus
\begin{equation}\label{eq:phi_b_0_expr}
  \psi(b_0) = -\sum_{b\in S\setminus\{b_0\}} \gamma_b \psi(b).
\end{equation}
Since $\pi_{b_0}$ is a symmetry of $\psi$, there exists an
invertible linear map $\tau:\Q^A \to \Q^A$ such that
$\psi(\pi_{b_0}(b))=\tau(\psi(b))$ for all $b \in B$. Applying
$\tau$ to \eqref{eq:phi_b_0_expr} gives
$$
\psi(e) = \tau(\psi(b_0)) = -\sum_{b\in S\setminus\{b_0\}} \gamma_b
\tau(\psi(b)) = -\sum_{b\in S\setminus\{b_0\}} \gamma_b
\psi(\pi_{b_0}(b)).
$$
Finally, $\Omega_2^c$ implies that $\pi_{b_0}(b)\notin E$ for all
$b\in S\setminus \{b_0\}$ and the lemma follows.
\end{proof}

We now derive Lemma~\ref{lemma:phi_FC_tame}.

\begin{proof}[Proof of Lemma~\ref{lemma:phi_FC_tame}]
We have by Lemma~\ref{lemma:phi_is_locally_correctable} that $\phi$
is $(\delta,s)$-locally correctable with $s=O(|A| \log(2c_2 |A|))$
and $\delta=\frac{1}{8s}$. We will establish
Lemma~\ref{lemma:phi_FC_tame} with $M=s/\sqrt{\delta}$.

Let us first prove the first item. Let
$\beta:=\left(\frac{1}{|B|}\sum_{b \in B}
|\ip{\phi(b),\theta}|^2\right)^{1/2}$ and set $E:=\{b \in B:
|\ip{\phi(b),\theta}| \ge \frac{\beta}{\sqrt{\delta}}\}$. Then we
must have that $|E| \le \delta |B|$. If $E$ is empty we are done.
Otherwise, since $\phi$ is $(\delta,s)$-locally correctable, for any
$e \in E$ we can express $\phi(e)$ as $\phi(e) = \sum_{b \notin E}
\gamma_b \cdot \phi(b)$ where $\sum |\gamma_b| \le s$. Hence in
particular,
$$
|\ip{\phi(e),\theta}| = |\sum_{b \notin E} \gamma_b
\ip{\phi(b),\theta}| \le \sum_{b \notin E} |\gamma_b| \cdot
|\ip{\phi(b),\theta}| \le s \cdot \frac{\beta}{\sqrt{\delta}}.
$$
The second item is very similar. Let
$\beta:=\left(\frac{1}{|B|}\sum_{b \in B} |r_b|^2\right)^{1/2}$ and
set $E:=\{b \in B: |r_b| \ge \frac{\beta}{\sqrt{\delta}}\}$. Again,
$|E|\le\delta|B|$ and thus any $e \in E$ can be expressed as
$\phi(e) = \sum_{b \notin E} \gamma_b \cdot \phi(b)$ with
$\|\gamma\|_1\le s$. Hence
$$
\ip{\phi(e),\theta} = \sum_{b \notin E} \gamma_b \ip{\phi(b),\theta}
= \sum_{b \notin E} \gamma_b (n_b+r_b),
$$
and in particular $r_e = \sum_{b \notin E} \gamma_b r_b \pmod{1}$
(where we mean that the modulo 1 maps $\R$ to $[-1/2, 1/2)$). Hence
\begin{equation*}
|r_e| \le \sum_{b \notin E} |\gamma_b| |r_b| \le s \cdot
\frac{\beta}{\sqrt{\delta}}.\qedhere
\end{equation*}
\end{proof}

\subsection{Estimating the Fourier transform near zero}
\label{subsec:fourier_near_zero}

We prove Lemma~\ref{lemma:estimate_fourier_near_zero} in this
subsection. Let $\theta \in \B_{R}(\eps) \subset \R^A$. Recall that
by Claim~\ref{cl:structure_fourier_X} we have that
$$
\hat{X}(\theta) = \prod_{b \in B} \left(1-p + p \cdot e^{2\pi i \ip{\phi(b),\theta}}\right).
$$
Let us shorthand $x_b = 2 \pi \ip{\phi(b),\theta}$. By our
assumptions that $\theta\in\B_R(\eps)$ and
Lemma~\ref{lemma:phi_FC_tame} we have
\begin{equation}\label{eq:x_b_bound}
  \max_{b\in B}|x_b| = 2\pi\max_{b\in B}|\ip{\phi(b),\theta}| \le 2\pi M \|\theta\|_R \le
  \frac{\pi}{4},
\end{equation}
where the last inequality follows from the assumption that
$\|\theta\|_R\le \eps\le\frac{1}{8M}$. Define the function $f:\R \to
\C$ given by $f(x):=1-p + p e^{ix}$. Then
\begin{equation}\label{eq:fourier_decompose_with_f}
\hat{X}(\theta) = \prod_{b \in B} f(x_b).
\end{equation}
\begin{claim} \label{claim:approx_f}
If $0 \le p \le 1$ and $|x|\le \frac{\pi}{4}$ then $f(x) =
\exp(ipx-\frac{1}{2}p(1-p) x^2+\delta(x))$ where $|\delta(x)| = O(p
|x|^3)$.
\end{claim}

\begin{proof}
Let $y=p(e^{ix}-1)$ so that $f(x)=1+y$. Our assumptions imply that
$|y|\le |e^{ix}-1|\le \sqrt{2-\sqrt{2}}<1$ so that
$\log(1+y)=y-y^2/2+O(|y|^3)$. Now, $y=ipx - px^2/2 + O(p|x|^3)$ and
$y^2=-p^2 x^2 + O(p^2|x|^3)$. Hence
\begin{equation*}
\log(f(x)) = ipx-\frac{1}{2}p(1-p)x^2+O(p|x|^3).\qedhere
\end{equation*}
\end{proof}
Applying Claim~\ref{claim:approx_f} to each term
in~\eqref{eq:fourier_decompose_with_f} we obtain
\begin{align}\label{eq:estimate_fourier_approx}
\hat{X}(\theta) &= \exp(2 \pi i \cdot p \cdot \sum_{b \in B} \ip{\phi(b),\theta}- 4 \pi^2 \cdot \frac{1}{2} p(1-p) \cdot \sum_{b \in B} \ip{\phi(b),\theta}^2 + \delta(\theta)) \nonumber \\
&= \exp(2 \pi i \cdot \ip{\Ex[X],\theta}- 2 \pi^2 \cdot \theta
^t\Sigma[X] \theta + \delta(\theta))
\end{align}
where we recall that $\Ex[X]=p \sum_{b \in B} \phi(b)$ and that
$\Sigma[X]_{a,a'}=p(1-p) \sum_{b \in B} \phi(b)_a \phi(b)_{a'}$, for
$a,a'\in A$. The error term is bounded by
\begin{equation}
|\delta(\theta)| = O(p \sum_{b \in B} |\ip{\phi(b),\theta}|^3).
\end{equation}
Applying \eqref{eq:x_b_bound} again yields
\begin{equation*}
  \sum_{b \in B} |\ip{\phi(b),\theta}|^3 \le \max_{b\in
B}|\ip{\phi(b),\theta}| \cdot \sum_{b \in B} |\ip{\phi(b),\theta}|^2
\le M\|\theta\|_R \cdot \|\theta\|_R^2 |B| = M\|\theta\|_R^3|B|,
\end{equation*}
and since $p|B|=N$ we can bound the error term by
$$
|\delta(\theta)| \le O(M \|\theta\|_R^3 N).
$$

\subsection{Bounding the Fourier transform far from $L$}
\label{subsec:fourier_far_L} We next bound the Fourier transform far
from $0$ in the $R$-norm, proving
Lemma~\ref{lemma:estimate_fourier_far_L}. Fix $\eps>0$ and
$\theta\in D\setminus\B_R(\eps)$. Our goal is to show that
$|\hat{X}(\theta)|$ must be small. Let us decompose
$\ip{\phi(b),\theta}=n_b+r_b$ where $n_b \in \Z$ and $r_b \in
[-1/2,1/2)$. Recall that by Claim~\ref{cl:structure_fourier_X} we
have that
\begin{equation}
\hat{X}(\theta) = \prod_{b \in B} \left(1-p + p \cdot e^{2\pi i
\ip{\phi(b),\theta}}\right) = \prod_{b \in B} \left(1-p + p \cdot
e^{2\pi i \cdot r_b}\right).
\end{equation}
We need two auxiliary claims.
\begin{claim}\label{claim:beta_lower_bound_by_eps}
$|\hat{X}(\theta)| \le \exp(-\frac{N}{|B|} \sum_{b \in B} r_b^2)\le
\exp(-\frac{N}{M^2}\max_{b\in B}
  r_b^2)$, where $M$ is defined in
  \eqref{eq:M_def}.
\end{claim}

\begin{proof}
It is simple to verify that for any $|x| \le 1/2$ and $p \le 1/2$,
$$
|1-p + p e^{2 \pi i x}| \le \exp(-p x^2).
$$
Hence
\begin{equation*}
|\hat{X}(\theta)| \le \exp(-p \sum_{b \in B} r_b^2) = \exp(-
\frac{N}{|B|} \sum_{b \in B} r_b^2).
\end{equation*}
The second inequality now follows from the second part of
Lemma~\ref{lemma:phi_FC_tame}.
\end{proof}

Recall that $c_3$ is the constant from our assumption that $V^\perp$
has a bounded integer basis in $\ell_1$. We next argue that if
$\max_{b\in B}|r_b|<\frac{1}{c_3}$ then the vector $(n_b)_{b \in B}$
belongs to the space $V$.

\begin{claim}\label{claim:decode_nb}
If $\max_{b\in B}|r_b| < 1/c_3$ then there exists $\alpha \in \Q^A$
such that $\ip{\phi(b),\alpha}=n_b$ for all $b \in B$.
\end{claim}
We remark that this is the only place in our proof where the
assumption that $V^\perp$ has a bounded integer basis in $\ell_1$ is
used.

\begin{proof}[Proof of Claim~\ref{claim:decode_nb}]
Assume to the contrary that no such $\alpha$ exists. Then $(n_b)_{b
\in B}$ does not belong to $V$ and hence it must violate some
constraint of $V^{\perp}$. However, by assumption, $V^{\perp}$ is
spanned by integer vectors of $\ell_1$ norm at most $c_3$. Hence,
there exists $\gamma \in \Z^B$, $\|\gamma\|_1 \le c_3$ such that
$$
\sum_{b \in B} \gamma_b n_b \ne 0.
$$
Since both $\gamma$ and $(n_b)$ are integer vectors, we must have that
$$
|\sum_{b \in B} \gamma_b n_b| \ge 1.
$$
However, we know that the vector $(n_b+r_b)_{b \in B}$ belongs to
$V$ (more precisely, to the span over $\R$ of the vectors in $V$).
Hence
$$
\sum_{b \in B} \gamma_b (n_b+r_b) = 0.
$$
Thus we conclude that
$$
|\sum_{b \in B} \gamma_b r_b| \ge 1.
$$
This is, however, impossible if $|r_b| < 1/c_3$ for all $b \in B$.
\end{proof}

We now conclude the proof of
Lemma~\ref{lemma:estimate_fourier_far_L}. There are two cases to
consider. Suppose first that $\max_{b\in B} |r_b|\ge \frac{1}{c_3}$.
Then Claim~\ref{claim:beta_lower_bound_by_eps} implies that
\begin{equation}\label{eq:hat_X_bound_first_case}
  |\hat{X}(\theta)|\le \exp\left(-\frac{N}{M^2c_3^2}\right)\quad\text{if $\max_{b\in B}
  |r_b|\ge\frac{1}{c_3}$}.
\end{equation}
Now assume instead that $\max_{b\in B} |r_b|<\frac{1}{c_3}$ and let
$\alpha \in \Q^A$ be as in Claim~\ref{claim:decode_nb}. By
definition, $\alpha$ belongs to the lattice $L$. It follows that
\begin{equation*}
  \|\theta-\alpha\|_R^2 = \frac{1}{|B|} \sum_{b\in B} \ip{\phi(b),\theta-\alpha}^2 =
  \frac{1}{|B|} \sum_{b\in B} r_b^2.
\end{equation*}
Thus, by the definition \eqref{eq:Voronoi_cell} of the Voronoi cell
$D$ and the fact that $\theta\in D$ and $\alpha\in L$, we deduce
that
\begin{equation*}
  \|\theta\|_R^2 \le \frac{1}{|B|} \sum_{b\in B} r_b^2.
\end{equation*}
Since $\theta\notin\B_R(\eps)$ we also have $\|\theta\|_R> \eps$ and
thus Claim~\ref{claim:beta_lower_bound_by_eps} shows that
\begin{equation}\label{eq:hat_X_bound_second_case}
  |\hat{X}(\theta)|\le \exp(-N\eps^2)\quad\text{if $\max_{b\in B}
  |r_b|<\frac{1}{c_3}$}.
\end{equation}
Taken together, \eqref{eq:hat_X_bound_first_case} and
\eqref{eq:hat_X_bound_second_case} prove
Lemma~\ref{lemma:estimate_fourier_far_L}.

\subsection{Proof of central limit theorem from auxiliary lemmas}
\label{sec:proof_of_main_thm_from_lemmas}

In this section we prove Theorem~\ref{thm:LCLT_with_basis}. We start
with a bound on the in-radius of $D$ in the $R$-norm.
\begin{claim}\label{claim:in_radius_D}
  If $\eps<\frac{1}{2M}$ then $\B_R(\eps)\subset D$.
\end{claim}
\begin{proof}
Let $0\neq\alpha\in L$. By definition, $\ip{\phi(b), \alpha}\in\Z$
for all $b\in B$. Since $\LL(\phi)$ is of full rank and $\alpha\neq
0$, there exists some $b\in B$ for which $|\ip{\phi(b), \alpha}|\ge
1$. It follows from Lemma~\ref{lemma:phi_FC_tame} that
$\|\alpha\|_R\ge \frac{1}{M}$. Since $\alpha$ is arbitrary, we
deduce from the definition~\eqref{eq:Voronoi_cell} of $D$ that
$\B_R(\eps)\subset D$ for any $\eps<\frac{1}{2M}$.
\end{proof}
Now fix $\lambda\in\LL(\phi)$ and recall from
Fact~\ref{fact:Fourier_inversion} that
\begin{equation}\label{eq:Fourier_inversion_X}
\Pr[X=\lambda] = \det(\LL(\phi))\int_D \hat{X}(\theta)
    e^{-2 \pi i \ip{\lambda,\theta}} d
    \theta\quad\forall\lambda\in\LL(\phi).
\end{equation}
Introduce a second random vector $Y\in\R^A$ having the Gaussian
distribution with mean $\Ex[X]$ and covariance matrix $\Sigma[X]$
(that is, with the same mean and covariance matrix as $X$). Recall
that the density function $f_Y$ of $Y$ equals
\begin{equation}\label{eq:Gaussian_density}
  f_Y(x):=\frac{\exp(-\frac{1}{2}(x-\Ex[X])^t\Sigma[X]^{-1}(x-\Ex[X]))}{(2\pi)^{\frac{|A|}{2}}\sqrt{\det{\Sigma[X]}}},
\end{equation}
and that the Fourier transform of $Y$ equals
\begin{equation}\label{eq:Gaussian_FT}
  \hat{Y}(\theta):=\Ex[e^{2\pi i\ip{Y,\theta}}] = e^{2\pi
  i\ip{\Ex[X],\theta}-2\pi^2\theta^t\Sigma[X]\theta}.
\end{equation}
Moreover, the Fourier inversion formula applied to $Y$ yields
\begin{equation}\label{eq:Fourier_inversion_Y}
  f_Y(x) = \int_{R^A} \hat{Y}(\theta)e^{-2\pi
  i\ip{x,\theta}}d\theta \quad\forall x\in\R^A.
\end{equation}
Theorem~\ref{thm:LCLT_with_basis} will follow by showing that
$\Pr[X=\lambda]$ approximately equals $\det(\LL(\phi))f_Y(\lambda)$.
Fix $0<\eps<\frac{1}{2M}$ whose exact value will be chosen later
(see \eqref{eq:eps_choice} and \eqref{eq:N_assumption}). Combining
\eqref{eq:Fourier_inversion_X}, \eqref{eq:Fourier_inversion_Y} and
Claim~\ref{claim:in_radius_D} we may write
\begin{align}\label{eq:Gaussian_approx}
  |\Pr[X=\lambda] &- \det(\LL(\phi))f_Y(\lambda)|\le\nonumber\\
  &\det(\LL(\phi))\Big(\underbrace{\int_{\B_R(\eps)}|\hat{X}(\theta)-\hat{Y}(\theta)|d\theta}_{=:I_1}
  + \underbrace{\int_{D\setminus\B_R(\eps)}|\hat{X}(\theta)|d\theta}_{=:I_2} +
  \underbrace{\int_{\R^A\setminus\B_R(\eps)}|\hat{Y}(\theta)|d\theta}_{=:I_3}\Big).
\end{align}
Our next lemma provides upper bounds on each of the above integrals.
\begin{lemma}\label{lemma:integral_bounds}
There exists a universal constant $c>0$ such that:
\begin{enumerate}
\item If $\eps\le \frac{1}{8M}$ and $M\eps^3 N\le c$ then
  \begin{equation*}
    I_1\le \frac{M|A|^{3/2}}{2\sqrt{N}(2\pi)^{\frac{|A|}{2}}\sqrt{\det(\Sigma[X])}}.
  \end{equation*}
  \item If $\eps\le\frac{1}{c_3 M}$ then
  \begin{equation*}
    I_2\le \frac{e^{-\eps^2 N}}{\det(\LL(\phi))}.
  \end{equation*}
  \item If $\eps^2 N\ge \frac{2|A|}{\pi^2}$ then
  \begin{equation*}
    I_3\le \frac{e^{-\frac{1}{4}\pi^2\eps^2
    N}}{(2\pi)^{\frac{|A|}{2}}\sqrt{\det(\Sigma[X])}}.
  \end{equation*}
  \end{enumerate}
\end{lemma}
\begin{proof}
  We start with the second item. By
  Lemma~\ref{lemma:estimate_fourier_far_L}, if $\eps\le \frac{1}{c_3
  M}$ and $\theta\in D\setminus\B_R(\eps)$ then
  $|\hat{X}(\theta)|\le \exp(-\eps^2 N)$. Hence,
  \begin{equation*}
    I_2 \le \vol(D)e^{-\eps^2 N} = \frac{e^{-\eps^2
    N}}{\det(\LL(\phi))}.
  \end{equation*}
  We continue with the third item. By \eqref{eq:Gaussian_FT} we have
  \begin{equation}\label{eq:I_3_value}
    I_3=\int_{\R^A\setminus\B_R(\eps)}
    e^{-2\pi^2\theta^t\Sigma[X]\theta}d\theta.
  \end{equation}
To evaluate the integral let $G$ be a standard multivariate Gaussian
random vector in $\R^A$ (with mean zero and identity covariance
matrix). Recalling that $\Sigma[X]$ is a positive definite matrix,
let $\Sigma[X]^{-1/2}$ be a symmetric positive definite matrix such
that $(\Sigma[X]^{-1/2})^2 = \Sigma[X]^{-1}$. It follows that
\begin{equation*}
Z:=\frac{1}{2\pi}\Sigma[X]^{-1/2}G
\end{equation*}
has a Gaussian distribution with mean zero and covariance matrix
$\frac{1}{4\pi^2}\Sigma[X]^{-1}$. Thus the density function of $Z$
is
\begin{equation}\label{eq:Z_density}
  f_Z(\theta) =
  (2\pi)^{\frac{|A|}{2}}\sqrt{\det(\Sigma)}e^{-2\pi^2\theta^t\Sigma[X]\theta}.
\end{equation}
Comparing with \eqref{eq:I_3_value} yields
\begin{align}\label{eq:I_3_to_Gaussian}
  I_3 =
  (2\pi)^{-\frac{|A|}{2}}\det(\Sigma)^{-\frac{1}{2}}\Pr[\|Z\|_R>\eps] \le
  (2\pi)^{-\frac{|A|}{2}}\det(\Sigma)^{-\frac{1}{2}}\Pr[\|G\|_2^2>2\pi^2
  \eps^2 N],
\end{align}
where in the last inequality we used that
\begin{equation}\label{eq:Z_to_R}
\|Z\|_R^2 = \frac{1}{|B|}Z^tRZ =
\frac{1}{4\pi^2|B|}G^t\Sigma[X]^{-1/2}R\Sigma[X]^{-1/2}G =
\frac{\|G\|_2^2}{4\pi^2 p(1-p)|B|}\le \frac{\|G\|_2^2}{2\pi^2N},
\end{equation}
recalling that $p|B|=N$ and our standing assumption that $p\le
\frac{1}{2}$. Now, The distribution of $\|G\|_2^2$ is chi-squared
with $|A|$ degrees of freedom. Observing that $\Ex e^{t\|G\|_2^2} =
(1-2t)^{-|A|/2}$ for $t<1/2$ and fixing $t=1/4$, Markov's inequality
yields for any $\rho\ge 4|A|$ that
$$
  \Pr[\|G\|_2^2 > \rho] \le \frac{\Ex e^{\|G\|_2^2/4}}{e^{\rho/4}} =
    2^{|A|/2} e^{-\rho/4} \le e^{-\rho/8}.
$$
Applying this result to \eqref{eq:I_3_to_Gaussian} and using our
assumption that $\eps^2 N\ge\frac{2|A|}{\pi^2}$ we have
\begin{equation*}
  I_3 \le \frac{e^{-\frac{1}{4}\pi^2\eps^2
  N}}{(2\pi)^{\frac{|A|}{2}}\sqrt{\det(\Sigma)}}.
\end{equation*}
Lastly, we verify the first item. Using
Lemma~\ref{lemma:estimate_fourier_near_zero} and
\eqref{eq:Gaussian_FT} we have
\begin{equation*}
  I_1 = \int_{\B_R(\eps)}e^{-2\pi^2
  \theta^t\Sigma[X]\theta}|e^{\delta(\theta)}-1|d\theta,
\end{equation*}
where $|\delta(\theta)|=O(M\|\theta\|_R^3 N)$. Our assumption that
$M\eps^3 N\le c$ for a sufficiently small $c$ implies that
$|e^{\delta(\theta)}-1|\le 2M\|\theta\|_R^3 N$ for
$\theta\in\B_R(\eps)$. Thus
\begin{align*}
  I_1 &\le 2M N \int_{\B_R(\eps)}e^{-2\pi^2
  \theta^t\Sigma[X]\theta}\|\theta\|_R^3 d\theta \le 2M N \int_{\R^A}e^{-2\pi^2
  \theta^t\Sigma[X]\theta}\|\theta\|_R^3 d\theta = \\
  &=\frac{2M
  N}{(2\pi)^{\frac{|A|}{2}}\sqrt{\det(\Sigma)}}\Ex\left[\|Z\|_R^3\right]\le \frac{2M}{(2\pi^2)^{3/2}\sqrt{N}(2\pi)^{\frac{|A|}{2}}\sqrt{\det(\Sigma)}}\Ex\left[\|G\|_2^{3}\right],
\end{align*}
where we used again the relations \eqref{eq:Z_density} and
\eqref{eq:Z_to_R}. Finally, observing that by Jensen's inequality,
$\Ex[\|G\|_2^3]\le (\Ex[\|G\|_2^4])^{3/4} =
(3|A|+|A|(|A|-1))^{3/4}\le 4^{3/4}|A|^{3/2}$ and that $\frac{2\cdot
4^{3/4}}{(2\pi^2)^{3/2}}\le \frac{1}{2}$ finishes the proof.
\end{proof}
We make the choice
\begin{equation}\label{eq:eps_choice}
  \eps := \sqrt{\frac{2|A|\log N}{N}}
\end{equation}
and the assumption
\begin{equation}\label{eq:N_assumption}
  N\ge C'\cdot c_2
  c_3^2|A|^6\log(2c_3|A|)^6
\end{equation}
for some universal constant $C'>0$ chosen sufficiently large for the
following calculations. It is simple to check that with these
choices the assumption $\eps<\frac{1}{2M}$ as well as all the
assumptions in the items of Lemma~\ref{lemma:integral_bounds} hold.
Thus, \eqref{eq:Gaussian_density}, \eqref{eq:Gaussian_approx} and
Lemma~\ref{lemma:integral_bounds} imply
\begin{align*}
  \Bigg|\Pr[X=\lambda] &-
  \frac{\det(\LL(\phi))}{(2\pi)^{\frac{|A|}{2}}\sqrt{\det{\Sigma[X]}}}e^{-\frac{1}{2}(\lambda-\Ex[X])^t\Sigma[X]^{-1}(\lambda-\Ex[X])}\Bigg|\le\\
  &\le \frac{\det(\LL(\phi))}{(2\pi)^{\frac{|A|}{2}}\sqrt{\det(\Sigma[X])}}\Bigg(\frac{M|A|^{3/2}}{2\sqrt{N}} + \frac{(2\pi)^{\frac{|A|}{2}}\sqrt{\det(\Sigma[X])}}{\det(\LL(\phi))}e^{-\eps^2
    N} + e^{-\frac{1}{4}\pi^2\eps^2
    N}\Bigg).
\end{align*}
To compare the middle summand in the right-hand side with the others
we use the following crude bounds on $\det(\LL(\phi))$ and
$\det(\Sigma[X])$. First, $\det(\LL(\phi))\ge 1$ since
$\LL(\phi)\subset\Z^A$. Second, since
$$
\Sigma[X]_{a,a'} = p(1-p) \sum_{b \in B} \phi(b)_a \phi(b)_{a'} \le
c_2^2 N,
$$
by Hadamard's inequality, we deduce that $\det(\Sigma[X]) \le
c_2^{2|A|}N^{|A|} |A|^{|A|/2}$. Thus, \eqref{eq:eps_choice} and
\eqref{eq:N_assumption} imply
\begin{equation*}
  \frac{(2\pi)^{\frac{|A|}{2}}\sqrt{\det(\Sigma[X])}}{\det(\LL(\phi))}e^{-\eps^2
  N} \le \frac{(2\pi c_2^2 N)^{\frac{|A|}{2}}|A|^{\frac{|A|}{4}}}{N^{2|A|}} \le
  \frac{1}{4N^{|A|/2}}
\end{equation*}
if the constant $C'$ in \eqref{eq:N_assumption} is large enough.
Since also $e^{-\frac{1}{4}\pi^2\eps^2 N}\le \frac{1}{4N^{|A|/2}}$
we finally conclude that
\begin{equation}\label{eq:main_estimate}
  \Pr[X=\lambda] = \frac{\det(\LL(\phi))}{(2\pi)^{\frac{|A|}{2}}\sqrt{\det{\Sigma[X]}}}\left(e^{-\frac{1}{2}(\lambda-\Ex[X])^t\Sigma[X]^{-1}(\lambda-\Ex[X])}+\delta\right)
\end{equation}
where $|\delta|\le \frac{M|A|^{3/2}}{\sqrt{N}}$. Recalling that
$\lambda$ is an arbitrary point in $\LL(\phi)$ and $N=p|B|$, we see
that we have proven Theorem~\ref{thm:LCLT_with_basis} in the case
$p\le \frac{1}{2}$.

We now get rid of the assumption $p\le \frac{1}{2}$. Fix $p\ge
\frac{1}{2}$, let $N:=p|B|$ and assume that
\begin{equation}\label{eq:N_assumption2}
  |B|-N \ge C'\cdot c_2 c_3^2|A|^6\log(2c_3|A|)^6
\end{equation}
holds. Recall that $X=\sum_{b\in B} T_b \phi(b)$ with the $\{T_b\}$
independent, identically distributed and satisfying
$\Pr[T_b=1]=1-\Pr[T_b=0]=p$. Let us temporarily write $\Pr_p, \Ex_p$
and $\Sigma_p[X]$ for the probability, expectation and covariance
matrix of $X$ with a given $p$. Denote $\phi(B):=\sum_{b\in
B}\phi(b)$. The fact that $X=\lambda$ if and only $\sum_{b\in B}
(1-T_b)\phi(b)=\phi(B)-\lambda$ implies that for any
$\lambda\in\LL(\phi)$, by \eqref{eq:main_estimate}, we have
\begin{align*}
  \Pr_p[X=\lambda] &= \Pr_{1-p}[X=\phi(B)-\lambda] =\\
  &=\frac{\det(\LL(\phi))}{(2\pi)^{\frac{|A|}{2}}\sqrt{\det{\Sigma_{1-p}[X]}}}\left(e^{-\frac{1}{2}(\phi(B)-\lambda-\Ex_{1-p}[X])^t\Sigma_{1-p}[X]^{-1}(\phi(B)-\lambda-\Ex_{1-p}[X])}+\delta\right)=\\
  &=\frac{\det(\LL(\phi))}{(2\pi)^{\frac{|A|}{2}}\sqrt{\det{\Sigma_{p}[X]}}}\left(e^{-\frac{1}{2}(\lambda-\Ex_{p}[X])^t\Sigma_{p}[X]^{-1}(\lambda-\Ex_p[X])}+\delta\right),
\end{align*}
with $|\delta|\le \frac{M|A|^{3/2}}{\sqrt{|B|-N}}$, as required. In
the last equality we used the facts that
$\Sigma_{1-p}[X]=\Sigma_p[X]$, $\Ex_{1-p}[X]=\phi(B)-\Ex_p[X]$ and
that $\mu^t D\mu = (-\mu)^t D(-\mu)$ for any vector $\mu$ and matrix
$D$. This establishes Theorem~\ref{thm:LCLT_with_basis} in full.

\subsection{Proof of main theorems}\label{sec:deduction_main_thms}
We now proceed to deduce Theorems~\ref{thm:main} and
\ref{thm:count_solutions} from the local central limit theorem,
Theorem~\ref{thm:LCLT_with_basis}. Assume the conditions of
Theorem~\ref{thm:main} and let $N$ be an integer satisfying
condition~\eqref{eq:N_condition}. We wish to estimate the number of
subsets $T\subset B$ of size $N$ satisfying
\begin{equation}\label{eq:prob_def_4}
  \frac{1}{|T|}\sum_{t\in T} f(t) = \frac{1}{|B|}\sum_{b\in B}
f(b)\quad\text{for all $f$ in $V$}.
\end{equation}
Define $p:=\frac{N}{|B|}$. Let $\{\phi_a:B \to \Q\}_{a \in A}$ be a
basis for $V$. Define the random subset $T\subset B$ and random
vector $X\in\Q^A$ as in Theorem~\ref{thm:LCLT_with_basis}. Since
\begin{equation*}
  X=\sum_{t\in T} \phi(t)
\end{equation*}
the event $X=\Ex[X]$ means
\begin{equation*}
  \sum_{t\in T} \phi(t) = p\sum_{b\in B} \phi(b) =
  \frac{N}{|B|}\sum_{b\in B} \phi(b).
\end{equation*}
Thus, since $\{\phi_a\}$, $a\in A$, is a basis for $V$, the event
$X=\Ex[X]$ is equivalent to
\begin{equation}\label{eq:X_Ex_conclusion}
  \sum_{t\in T} f(t) = \frac{N}{|B|}\sum_{b\in B}
f(b)\quad\text{for all $f$ in $V$}.
\end{equation}
Now, by assumption, the constant function $h\equiv 1$ belongs to
$V$. Thus, on the event $X=\Ex[X]$ we have
\begin{equation}\label{eq:constant_assumption_usage}
  |T| = \sum_{t\in T} h(t) = \frac{N}{|B|}\sum_{b\in B} h(b) = N.
\end{equation}
Comparing \eqref{eq:prob_def_4}, \eqref{eq:X_Ex_conclusion} and
\eqref{eq:constant_assumption_usage} we see that the event
$X=\Ex[X]$ is equivalent to the event that $|T|=N$ and
\eqref{eq:prob_def_4} holds. Now, denoting by $\alpha_N$ the number
of subsets $T\subset B$ of size $N$ for which \eqref{eq:prob_def_4}
holds it follows that
\begin{equation}\label{eq:X_alpha_relation}
  \Pr[X=\Ex[X]] = \alpha_N p^N (1-p)^{|B|-N}.
\end{equation}
Finally, the divisibility assumption implies that
\begin{equation*}
  \Ex[X] = \frac{N}{|B|}\sum_{b\in B}\phi(b)\in\LL(\phi).
\end{equation*}
Thus we may substitute $\lambda=\Ex[X]$ in
Theorem~\ref{thm:LCLT_with_basis} to obtain
\begin{equation}\label{eq:prob_central_point1}
  \Pr[X=\Ex[X]] =
  \frac{\det(\LL(\phi))}{(2\pi)^{\frac{|A|}{2}}\sqrt{\det{\Sigma[X]}}}(1+\delta),
\end{equation}
with $|\delta|\le \frac{C\dim(V)^3
(\log(2c_2\dim(V)))^{3/2}}{\sqrt{\min(N,|B|-N)}}$. Comparing
\eqref{eq:X_alpha_relation} and \eqref{eq:prob_central_point1}
proves the assertion of Theorem~\ref{thm:count_solutions}. Lastly,
Theorem~\ref{thm:main} follows upon noting that the assumption
\eqref{eq:N_condition} implies that $|\delta|\le \frac{1}{2}$, so
that $\Pr[X=\Ex[X]]>0$ and hence $\alpha_N>0$.

\begin{subsection}{Basis-free formulation of local central limit
theorem}\label{sec:LCLT_basis_free} In this section we describe an
equivalent ``basis-free'' version of our local central limit
theorem, Theorem~\ref{thm:LCLT_with_basis}. The theorem is a
high-dimensional, lattice, local central limit theorem with a rate
of convergence estimate involving only universal constants.

Recall the parameter $\rho(V)$ of the vector space $V$ introduced in
\eqref{eq:rho_formula}. We introduce a second parameter of $V$, a
non-negative definite form $\ip{\cdot,\cdot}_V$ on $\Q^B$. As is the
case for $\rho(V)$, it is easiest to define $\ip{\cdot,\cdot}_V$ via
a choice of basis for $V$ but we stress that it is independent of
this choice. If $\phi:B\to\Q^A$ is such that the vectors $(\phi_a)$,
$a\in A$, form a basis for $V$, we define
\begin{equation*}
  \ip{\gamma_1,\gamma_2}_V :=
  \gamma_1^t\phi(\phi^t\phi)^{-1}\phi^t\gamma_2\qquad(\gamma_1,\gamma_2\in\Q^B).
\end{equation*}
In this definition, $\phi$ is regarded as a $B\times A$ matrix with
columns $\{\phi_a\}$. The matrix $\phi(\phi^t\phi)^{-1}\phi^t$
represents the orthogonal projection operator from $\Q^B$ (with the
standard basis and inner product) to $V$. We denote the semi-norm
induced from $\ip{\cdot,\cdot}_V$ by $\|\cdot\|_V$,
\begin{equation*}
  \|\gamma\|_V:=\sqrt{\ip{\gamma,\gamma}_V}\qquad(\gamma\in\Q^B),
\end{equation*}
so that $\|\gamma\|_V$ is the length of the orthogonal projection of
$\gamma$ to $V$. Finally, we denote by ${\bf 1}$ the identically one
vector in $\Q^B$.

\begin{thm}[Basis-free formulation of local central limit theorem]\label{thm:LCLT_without_basis}
There exists a constant $C>0$ such that the following is true. Let
$B$ be a finite set and let $V$ be a linear subspace of functions
$f:B \to \Q$. Assume that the following conditions hold for some
integers $c_2,c_3 \ge 1$,
\begin{enumerate}
\item Boundedness of $V$: $V$ has a $c_2$-bounded integer basis in $\ell_{\infty}$.
\item Boundedness of $V^{\perp}$: $V^{\perp}$ has a $c_3$-bounded integer basis in $\ell_{1}$.
\item Symmetry: for any $b_1,b_2 \in B$ there exists a symmetry $\pi$ of $V$ satisfying $\pi(b_1)=b_2$.
\end{enumerate}
Let $0<p<1$ and form a random subset $T\subset B$ by taking each
element of $B$ into $T$ independently with probability $p$. If
\begin{equation*}
  \min(p|B|, (1-p)|B|) \ge C \cdot c_2 c_3^2\dim(V)^6\log(2c_3\dim(V))^6
\end{equation*}
then for every $\gamma\in\Z^B$ the probability of the event
\begin{equation}\label{eq:LCLT_gamma_event}
\sum_{t\in T} f(t) = \sum_{b\in B} \gamma_b f(b)\quad\text{for all
$f$ in $V$}
\end{equation}
equals
\begin{equation}\label{eq:LCLT2_prob_estimate}
\frac{\rho(V)}{(2\pi
p(1-p))^{\frac{\dim(V)}{2}}}\Bigg(\exp\bigg(-\frac{\|\gamma-p\cdot{\bf
1}\|_V^2}{2p(1-p)}\bigg)+\delta(\gamma)\Bigg)
\end{equation}
with $|\delta(\gamma)|\le \frac{C\dim(V)^3
(\log(2c_2\dim(V)))^{3/2}}{\sqrt{\min(p|B|,(1-p)|B|)}}$.
\end{thm}
It is important to emphasize that one must take $\gamma\in\Z^B$,
rather than $\gamma\in\Q^B$ in the theorem, analogously to the
restriction that $\lambda\in\LL(\phi)$ in
Theorem~\ref{thm:LCLT_with_basis}. However, since $V$ typically has
dimension strictly less than $|B|$, it is possible that for some
$\gamma\in\Q^B\setminus\Z^B$ there exists another $\gamma'\in\Z^B$
such that
\begin{equation}\label{eq:gamma_gamma_prime_equality}
  \sum_{b\in B} \gamma_b f(b) = \sum_{b\in B} \gamma_b' f(b)\quad\text{for all
$f$ in $V$}.
\end{equation}
Indeed, this is exactly the scenario we face in our main theorems.
There, we are interested in the case that $\gamma_b=\frac{N}{|B|}$
for all $b\in B$, a vector which is not in $\Z^B$. The divisibility
condition in Theorem~\ref{thm:main} exactly ensures that for this
vector there exists some $\gamma'\in\Z^B$ such that
\eqref{eq:gamma_gamma_prime_equality} holds.

Let us say that a vector $\gamma\in\Q^B$ has an \emph{integer
representation by $\gamma'\in\Z^B$} if
\eqref{eq:gamma_gamma_prime_equality} holds. It is not difficult to
check that in this case $\|\gamma-p\cdot{\bf 1}\|_V =
\|\gamma'-p\cdot1\|_V$. Thus, our theorem remains true as stated if
the restriction that $\gamma\in\Z^B$ is replaced by the condition
that $\gamma\in\Q^B$ and has an integer representation. Moreover, it
is evident that if $\gamma$ has no integer representation then the
probability of \eqref{eq:LCLT_gamma_event} is zero, since the left
hand side of \eqref{eq:LCLT_gamma_event} exactly provides an integer
representation for $\gamma$.

We now briefly explain the equivalence of
Theorem~\ref{thm:LCLT_with_basis} and
Theorem~\ref{thm:LCLT_without_basis}. Suppose that $\{\phi_a\}$,
$a\in A$, form a basis for $V$, define $\phi:B\to\Q^A$ by $\phi(b)_a
= \phi_a(B)$ and regard $\phi$ as a $B\times A$ matrix. If
$\lambda\in\LL(\phi)$ then, by definition, there exists some
$\gamma\in\Z^B$ such that
\begin{equation}\label{eq:lambda_gamma_relation}
\lambda = \sum_{b\in B}\gamma_b \phi(b) = \phi^t\gamma.
\end{equation}
It is then straightforward to check that for a subset $T\subset B$
condition \eqref{eq:LCLT_gamma_event} is equivalent to
\begin{equation}\label{eq:lambda_event}
  \sum_{t\in T} \phi(t) = \lambda.
\end{equation}
Conversely, given $\gamma\in\Z^B$ we may define
$\lambda\in\LL(\phi)$ by \eqref{eq:lambda_gamma_relation} and
observe again that \eqref{eq:LCLT_gamma_event} is equivalent to
\eqref{eq:lambda_event}. Thus, to see the equivalence of the two
theorems, it suffices to show that the main terms in the probability
estimates \eqref{eq:LCLT_prob_estimate} and
\eqref{eq:LCLT2_prob_estimate} are equal when the relation
\eqref{eq:lambda_gamma_relation} holds. This follows from the
definitions of $\rho(V)$, the definitions of $\Ex[X]$ and
$\Sigma[X]$ in Subsection~\ref{sec:LCLT_with_basis} and the
observation that under \eqref{eq:lambda_gamma_relation} we have
\begin{align*}
  \|\gamma-p\cdot{\bf 1}\|_V^2 &= (\gamma-p\cdot{\bf
  1})^t\phi(\phi^t\phi)^{-1}\phi^t(\gamma-p\cdot{\bf
  1}) =\\
  &= (\lambda - \Ex[X])^t(\phi^t\phi)^{-1}(\lambda-\Ex[X]) =
  p(1-p)(\lambda-\Ex[X])^t\Sigma[X]^{-1}(\lambda-\Ex[X]).
\end{align*}
\end{subsection}

\section{Summary and open problems}
\label{sec:summary}

Our main theorem guarantees the existence of a small subset $T
\subset B$ for which \eqref{eq:prob_def} holds. The conditions we
require are boundedness, divisibility and symmetry. In many natural
scenarios it is easy to guarantee that $V$ has a bounded integer
basis in $\ell_{\infty}$, the divisibility and the symmetry
condition, and the condition which seems hardest to verify is that
$V^{\perp}$ has a bounded integer basis in $\ell_1$. In particular,
the following question captures much of the difficulty. Let $G$ be a
group that acts transitively on a set $X$. A subset $T \subset G$ is
\emph{$X$-uniform} (or an \emph{$X$-design}) if it acts on $X$
exactly as $G$ does. That is, for any $x,y \in X$,
$$
\frac{1}{|T|} |\{g \in T: g(x)=y\}| = \frac{1}{|G|} |\{g \in G:
g(x)=y\}| = \frac{1}{|X|}.
$$
In our language we may take $B=G$ and $V$ to be the space spanned by
all functions $\phi_{(x,y)}:B\to\{0,1\}$ of the form
$\phi_{(x,y)}(b)=\mathbf{1}_{\{b(x)=y\}}$ for $x,y\in X$. Then $T$
is $X$-uniform if and only if \eqref{eq:prob_def} holds. We have
given a bounded integer basis for $V$ in $\ell_{\infty}$, and also
by definition the symmetry condition holds. The other conditions are
less clear. One may still speculate that:

\begin{conj} Let $G$ be a group that acts transitively on a set $X$. Then
there exists an $X$-uniform subset $T \subset G$ such that $|T| \le |X|^c$
for some universal constant $c>0$.
\end{conj}

A second question is whether one can apply our techniques to get
\emph{minimal} objects. Recall that the size of the objects we
achieve is only minimal up to polynomial factors. For example, can
one use these methods to show the existence of a Steiner system
(i.e., a $t$-design with $\lambda=1$)? A major open problem of a
similar spirit is the existence of Hadamard matrices of all orders
$n=4m$, or equivalently, $2$-$(4m-1, 2m-1, m-1)$ designs. Empirical
estimates for $n \le 32$ suggest that there are $\exp(O(n \log n))$
Hadamard matrices of order $n = 4m$. Since there are so many of
them, and since the logarithm of their number grows at a regular
rate, we suspect that they exist for some purely statistical reason.
However, the Gaussian local limit model seems to be false for
Hadamard matrices interpreted as $t$-designs; it does not accurately
estimate how many there are.

A third question is whether there exists an algorithmic version
of our work, similar to the algorithmic Moser~\cite{Moser2009} and
Moser-Tardos~\cite{MoserTardos2010} versions of the Lov\'{a}sz local
lemma~\cite{ErdosLovasz73}, and the algorithmic Bansal~\cite{Bansal2010}
and Lovett-Meka~\cite{LovettMeka12} versions of the six standard deviations method of Spencer~\cite{Spencer1985}.
If an efficient randomized algorithm of our method were found, then we
could no longer indisputably claim that we have a low-probability version
of the probabilistic method.  On the other hand it would be strange,
from the viewpoint of computational complexity theory, if low-probability
existence can always be converted to high-probability existence.  Maybe our
construction is fundamentally a low-probability construction.

It is also of interest to extend our results to continuous setups,
with one representative example being that of spherical designs (see
\cite{SeymourZaslavsky84}, \cite{BondarenkoRadchenkoViazovska},
\cite{kane2015small}, \cite{Gilboa2016Chebyshev} and references
within).

{\bf Acknowledgements.} We thank David Soudry and Gady Kozma for
useful remarks on the representation theory of the symmetric group.
We thank Alexander Barvinok, Peter Keevash and Brendan McKay for
valuable discussions of the connections between their work and ours.

\bibliographystyle{hamsalpha}
\bibliography{designs}

\end{document}